\documentclass[12pt]{article}
\usepackage[letterpaper, portrait, left = 1in, right = 1in, top = 1.2in, bottom=1.5in]{geometry} 
\usepackage[table,xcdraw]{xcolor}
\usepackage{amsmath}
\usepackage{amssymb}
\usepackage{mathtools}
\usepackage{amsthm}
\usepackage{hyperref}
\usepackage{enumitem}
\usepackage{natbib}
\RequirePackage{bm}
\usepackage[linesnumbered,vlined,ruled,noend]{algorithm2e}
\usepackage{wrapfig}
\usepackage{quiver}
\usepackage{tikz-cd}
\usepackage{charter}

\theoremstyle{plain}
\newtheorem{theorem}{Theorem}[section]
\newtheorem{proposition}{Proposition}
\newtheorem{lemma}{Lemma}
\newtheorem{corollary}{Corollary}
\theoremstyle{definition}
\newtheorem{definition}{Definition}

\theoremstyle{remark}

\newcommand{\cA}{{\mathcal{A}}}

\newcommand{\cO}{{\mathcal{O}}}   \newcommand{\cP}{{\mathcal{P}}}
   
   \newcommand{\cT}{{\mathcal{T}}}
\newcommand{\cU}{{\mathcal{U}}}

\newcommand\N{\ensuremath{\mathbb{N}}}

\newcommand\R{\ensuremath{\mathbb{R}}}
\newcommand\Z{\ensuremath{\mathbb{Z}}}

\newcommand\Q{\ensuremath{\mathbb{Q}}}
\newcommand\C{\ensuremath{\mathbb{C}}}

\renewcommand{\phi}{\varphi}

\newcommand{\st}{%
  \nonscript\;
  \ifnum\currentgrouptype=16
    \;\middle|\;
  \else
    \;|\;
  \fi
  \nonscript\;}

\newcommand{\abs}[1]{\left| #1 \right|}
\newcommand{\inner}[2]{\left\langle #1, #2 \right\rangle}

\newcommand{\ds}{\displaystyle}

\DeclareMathOperator{\sgn}{sgn}

\let\oldend\endlinechar
\renewcommand{\endlinechar}{\oldend}

\newcommand{\divides}{\mathbin{|}}
\newcommand{\set}[1]{\ensuremath{\left\{#1\right\}}}
\newcommand{\sett}{\coloneqq}

\renewcommand{\epsilon}{\varepsilon}
\newcommand{\fa}{~\forall~}
\usepackage{titletoc}
\newcommand{\sint}{\sin\theta}
\newcommand{\cost}{\cos\theta}

\newcommand{\lp}{\left(}
\newcommand{\rp}{\right)}

\newcommand{\pa}[1]{\lp#1\rp}
\newcommand{\norm}[2]{
\left\lVert #1 \right\rVert_{#2}
}

\usepackage{etoolbox}
\newtoggle{isOR}
\togglefalse{isOR}

\newtoggle{isORtext}
\togglefalse{isORtext} 

\usepackage[utf8]{inputenc} 
\usepackage[T1]{fontenc}    
\usepackage{hyperref}       
\usepackage{url}            
\usepackage{booktabs}       
\usepackage{amsfonts}       
\usepackage{nicefrac}       
\usepackage{microtype}      
\usepackage{xcolor}         

\usepackage{caption}  

\newcommand{\informsproof}{}
\newcommand{\informsendofproof}{}

\title{Temporal Robustness in Discrete Time Linear Dynamical Systems}

\author{%
  Nilava Metya\\
  Rutgers University\\
  New Brunswick, NJ \\
  \texttt{nilava.metya@rutgers.edu} \\
  \and
  Ankit Shah\\
  Indiana University\\
  Bloomington, Indiana\\
  \texttt{ankit@iu.edu}\\
  \and 
  Arunesh Sinha\\
  Rutgers University\\
  New Brunswick, NJ \\
  \texttt{arunesh.sinha@rutgers.edu} \\
}

\date{}

\begin{document}

\maketitle

\begin{abstract}
Discrete time linear dynamical systems, including Markov chains, have found many applications including in security settings such as in cybersecurity operations center (CSOC) management and in managing health risks. However, in these two scenarios, there is uncertainty about the time horizon for which the system runs. This creates uncertainty about the cost (or reward) incurred based on the state distribution when the system stops. Given past data samples of how long a system ran, we theoretically analyze the cost incurred at the stop of the system as a distributional robust cost estimation task in a Wasserstein ambiguity set. Towards this, we show an equivalence between a discrete time Markov Chain on a probability simplex and a global asymptotic stable (GAS) discrete time linear dynamical system, allowing us to base our study on a GAS system only. Then, we provide various polynomial time algorithms and hardness results for different cases in our theoretical study, including a novel proof of a fundamental result about Wassertein distance based polytope. We experiment with real world data in CSOC domain and prior data in health domain to reveal the benefits of our model and approach.
\end{abstract}

\section{Introduction}
\label{introduction}
Discrete time linear dynamical systems, prominently Markov chains, have found many uses in machine learning \citep{hao2018learning} as well as other sciences~\citep{husic2018markov}, with applications in social networks~\citep{yang2011like}, health~\citep{yaesoubi2011generalized,sato2010markov}, and various other domains. There is a huge body of work on understanding the properties of such models~\citep{howard2012dynamic}, including asymptotic behavior and convergence rate analysis. Yet, in many applications the challenge lies not only in characterizing long-run properties, but also in reasoning about system performance within a fixed wall-clock horizon, where the effective number of discrete steps is itself uncertain due to variability in human behavior, external conditions, or contextual factors. \textcolor{black}{Addressing this horizon uncertainty is crucial for robust short-term prediction, risk estimation, and cost assessment, particularly in domains central to national and human security, such as cybersecurity operations and public health. We illustrate this challenge in both contexts below.} 

A Cybersecurity Operations Center (CSOC) is a centralized facility where analysts continuously monitor, investigate, and respond to security alerts. A CSOC is critically important as the frontline defense for protecting organizational assets, ensuring resilience against cyber threats, and maintaining trust in digital infrastructure. Most large organizations, including national defense agencies, operate CSOCs as integral components of their security infrastructure. 
Within these centers, analysts contend with a continuous stream of alerts \textcolor{black}{generated by intrusion detection systems. This alert-handling process} has been modeled as a Markov chain, where each state corresponds to the number of alerts in an analyst’s queue. During the regular work shift, alerts arrive at rate $\lambda$ (of a Poisson process) and are processed at rate $\mu$ (mean of a uniform distribution)~\citep{shah2018methodology, 10.1145/3173457, 10.1145/3377554}, producing a transition matrix $M’$ over queue lengths. After a standard workday shift (e.g., eight hours, 960 steps of 30 seconds each), the system enters an \emph{overtime Markov chain} given by $M$, where arrivals stop (since new alerts are assigned to the next shift), but the backlog continues to be processed. The stop time of this overtime chain is uncertain, depending on external factors such as analyst availability, for which some past data may be available. From past data on overtime behavior, one can estimate a nominal but uncertain probability distribution over the possible stopping times. At termination, the cost depends jointly on the backlog distribution and the number of overtime steps, creating a direct need for robust estimation under uncertain time horizons to ensure resilient and cost-effective cybersecurity operations.

A similar phenomenon arises in health security, where stochastic models, such as Markov chains, of infectious disease spread, such as those used during a pandemic (natural-disaster-driven threats disrupting access to healthcare, food, and energy resources), inform resource allocation and resilience planning~\citep{10.1093/oso/9780192896087.003.0022,yaesoubi2011generalized,dudel2020estimating}. For example, in an SIR (Susceptible-Infected-Recovered) model, the health state of each individual evolves via a transition matrix $M$ over a time step, where the time step is chosen by domain experts as a time interval of length $l$. An important consideration is what time length $l$ should be chosen for as a time step~\citep{yaesoubi2011generalized}. Often, this time length $l$ is fixed by experts (e.g., one day) and $M$ estimated from repeated observation over these time intervals of length $l$. Thus, $M$ captures the effect of ``average'' interaction (that enables disease spread) over this time interval. However, when applied for any population, the intensity of interaction can vary depending on various human factors such as norms (masking, etc.) and social events at that time. This has an effect that $M$ might describe interaction over a time interval of length $l'$ where $l'$ is a random quantity dependent on group behavior. This results in a problem where the number of timesteps in a given time frame (e.g., one week) becomes uncertain, for which some data might be available from domain expert inputs. To handle the uncertainty, one may be interested in the worst case state of disease spread in a week in order to robustly estimate the cost of disease spread and make critical decisions about social distancing or masking. 

These examples illustrate a broader methodological gap: there is a fundamental uncertainty in the number of time steps the system evolves. Existing approaches often introduce an absorbing state to model termination~\citep{freedman2012markov}, assuming a known probability distribution over horizon lengths. However, this assumes precise knowledge of the distribution, which may be unavailable or unreliable when only a \emph{finite number of samples} exist. In contrast, we aim for robustness by forming uncertainty or ambiguity sets based on prior data samples (or expert inputs) of how many timesteps a Markov chain could run.


To address this, we propose a robustness framework for discrete-time linear dynamical systems with uncertain horizons.
Toward our goal, we first prove a fundamental result showing an equivalence between a discrete time Markov chain on a probability simplex and a \emph{global asymptotic stable} (GAS) discrete time linear dynamical system with the origin $\bm 0$ as fixed point. In particular, a GAS system satisfies some properties, such as Lyapunov stability (Definition~\ref{def:gas}). The convergence to $\bm 0$ allows for easier notation and proof for our theory results for a GAS system and the equivalence provides an all-encompassing framework applicable for Markov chains without loss of generality (WLOG).

Then, we specify a robust cost estimation problem in a GAS system with an uncertain time horizon but given data samples of the number of time steps from past runs. We formulate this as a distributional problem with Wasserstein ambiguity set around a nominal distribution formed using the given time step samples. We call this the distributional robust cost estimation \eqref{eq:dre} problem and a special case with unbounded ambiguity set is the robust cost estimation~\eqref{eq:re} problem. We consider two cases to solve: (1) \emph{finite support} of the nominal distribution and (2) \emph{infinite support} of the nominal distribution. We also show that our approach can jointly handle addtional uncertainty in the inital state.
We list all our contributions next:
\begin{enumerate}
    \item A fundamental equivalence between a discrete time Markov Chain on a probability simplex and a GAS system.
    \item  A fast polynomial time algorithm, named \emph{Small and Big Strides}, (Algorithm~\ref{alg:robustfinite}) for evaluating the GAS state over finitely many time steps.
    \item A detailed proof of the structure of Wasserstein-1 polytope on the probability simplex, which combined with Algorithm~\ref{alg:robustfinite} results in a polynomial time solution for Wasserstein ambiguity set in the finite support case.
    \item  A NP hardness result for the infinite support case, and then a non-polynomial time algorithm for~\eqref{eq:re} problem and an approximation for the~\eqref{eq:dre} problem in this case.
    \item Experiments in a real world CSOC problem and in health security domain validate the usefulness of our model and approach.
\end{enumerate}

The rest of the paper is structured as follows: we situate the problem within relevant literature in Section~\ref{sec:related}; next, we provide some background, and basic definitions in Section~\ref{sec:background}; next, we establish the conversion of a Markov chain to a GAS system with commuting identities and state recovery in Section~\ref{sec:relation}; next, we formulate the {\sf DRCE} and {\sf RCE} problem in Section~\ref{sec:statement}; next, we we present the core of our paper with all the methodological results in Section~\ref{sec:approach}; next, we demonstrate scalability, and conservative robustness in CSOC security monitoring and SIR disease spread setting in Section~\ref{sec:experiments}; and we conclude with key takeaways and future directions in Section~\ref{sec:conclusion}. All missing and full proofs are in Appendix~\ref{sec:proofs}.

\section{Related Work}\label{sec:related}
\textit{Distributional Robustness}: A very widely studied topic in robust optimization literature~\citep{kuhn2025distributionally,xu2012distributional}, distributional robustness has recently found usage in machine learning~\citep{KuhnWasserstein,duchi2019variance,lee2018minimax}, with applications in adversarial learning~\citep{sinha2018certifying,moayeri2022explicit,kumarprovable,bai2024wasserstein}, robust decision making~\citep{staib2019distributionally,bose2022scalable},and reinforcement learning~\citep{liu2022distributionally,shi2024curious,zhangdistributionally}. Our work is situated in the same space but 
as far as we know this is a first work on distributional robustness for uncertain time horizon in linear dynamical systems. 
We do not solve a conventional distributional robust optimization~\citep{rahimian2022frameworks}, which also involves choosing an action (or decision, or control, or sometimes weights of a neural network) in an outer decision-making problem. We perform the inner distributional robust estimation part of a typical distributional robust optimization~, but in the challenging situation of uncertain time horizon for a dynamically evolving state distribution.
There are other works also that tackle the inner problem, and call it an evaluation of performance or risk~\citep{blanchet2019quantifying}; this work presents a general dual formulation of such problems. However, our focus here is on fast computational methods (considering the dynamics) and the dual form does not help in faster computation in our problem. Thus, this dual form is only a structural result; our detailed analysis of why it does not yield computational improvements, which requires the problem setup and notation, is deferred to the appendix. 

\textit{Linear Dynamical Systems and Markov Chains}: There are many works in linear dynamical systems, some recent ones focus on learning such systems~\citep{hazan2018spectral,pmlr-v97-sarkar19a,pmlr-v202-bakshi23a,tyagi2023learning}. Markov chains have been used very widely, such as in choice modeling~\citep{blanchet2016markov}, in machine learning such as in MCMC approaches~\citep{jones2022markov}, and for our two problem domains as evident from citations in the introduction. There is also work on uncertainty in transition dynamics in Markov chains~\citep{rhee2023lyapunov}. However, as stated in the introduction, different from any work that we know of, we aim for distributional robustness to tackle the uncertainty in time horizon.

\section{Background and Preliminaries}
We present basic notations and background in linear algebra, which can be found in textbooks~\citep{meyer2023matrix}. We also present definitions for finite space and discrete time linear dynamical systems.

\textit{Notation}: $[n]$ denotes $\{1, \ldots, n\}$. We use bold notation, e.g., $\bm x$, to denote vectors. $\bm x^\top$ is the transpose and $\text{dim} (\bm x)$ is the dimension of vector $\bm x$.  $\bm e_i$ denotes the $i^{\text{th}}$ standard basis vector with $1$ at position $i$ and $0$ elsewhere (dimension of the ambient space will be clear from context or mentioned if required), and $\bm 1\sett\sum_i\bm e_i$. Denote $H_{n} \sett \set{\bm x\in \R^{n} \st \bm 1^\top \bm x=0}$, $\Delta_{n}\sett \set{\bm x\in\R^n\st x_i\ge 0\fa i\in[n], \bm 1^\top \bm x=1}$ the probability simplex and $\overline \Delta_{n}\sett \set{\bm x\in\R^n\st \bm 1^\top \bm x=1}$ the extended probability simplex which is the same as $H_{n}+\frac1n\bm 1$. 
For a set $S \subseteq \N$, a distribution on $S$ is a function $q : S\to [0,1]$ such that $\sum_{s\in S}q(s) = 1$. These are the same as vectors $\bm q\in \R^S$ with non-negative entries that sum to $1$. We will interchangeably use $q(s) = q_s$. The set of all distributions on $S$ is denoted as $\Delta(S)$. 
$M^\top$ denotes matrix $M$ transpose and $M_{ij}$ is the entry in row $i$ and column $j$. 
$\rho(M)$ is the spectral radius of $M$, i.e., $\rho(M) = \max\{|\lambda_1|, \ldots, |\lambda_n|\}$ for the $n$ eigenvalues $\lambda_i$'s. $I_n$ denotes the identity matrix of size $n \times n$. 
Any complex number $a + ib$ can be written in Euler notation as $re^{i\theta} = r (\cos \theta + i \sin \theta)$, where $r = \sqrt{a^2 + b^2}\geq 0$ is the magnitude. Given non-zero $b$, we must have $\theta \in (0, \pi) \cup (\pi, 2\pi)$. 

\textit{Real Jordan normal form}: Any square matrix $M \in \R^{n\times n}$ can be written as $M = P J P^{-1}$ where $P \in \R^{n \times n}$, i.e., matrix with real valued entries, and $J$ is a real valued matrix in Jordan form as described next. First, the eigenvalues of $M$ can be complex or real valued. Complex eigenvalues appear in pairs, which we denote by $a+ib$ and $a-ib$. Suppose $M$ has $q$ pairs of complex eigenvalues and $p$ real eigenvalues ($q$ or $p$ could be $0$). Assume that the eigenvalues are distinct; if two eigenvalues are same, perturb one by an arbitrarily small amount. See Appendix~\ref{sec:misc} for details. Let the real eigenvalues be given by $\lambda_1 > \ldots > \lambda_p$. For each complex pair of eigenvalues $a_j \pm ib_j$, define the Jordan block: 
$\begin{bmatrix}
a_j & -b_j \\
b_j & a_j
\end{bmatrix}
$ or equivalently in the Euler notation as 
$r_j \begin{bmatrix}
\cos \theta_j & -\sin \theta_j \\
\sin \theta_j & \cos \theta_j
\end{bmatrix}
 = r_j J_j$. Then, the Jordan matrix $J$ is a \emph{block diagonal matrix} given as
\begin{equation}\label{jordan}
J = \text{diag}(r_1 J_1, \ldots,r_q J_q,  \lambda_1, \ldots, \lambda_p) \; ,\end{equation}
where $\text{diag}(\cdot)$ creates a square diagonal matrix with the inputs on the main diagonal and zeros everywhere else.
Matrix exponentiation takes a rather convenient form with Jordan form: $M^k = P J^k P^{-1}$, where $J^k$ has the same structure as $J$: $J^k = \text{diag}\left(r_1^k J_1^k, \ldots, \lambda_p^k\right)$. In particular, it is straightforward to check that $J_j^k = \begin{bmatrix}
\cos k\theta_j & -\sin k\theta_j \\
\sin k\theta_j & \cos k\theta_j
\end{bmatrix}$.

\textbf{Discrete Time Linear Dynamical Systems. } \label{sec:background}
A discrete-time dynamical system on a metric space $(X,d)$ is given by a function $f:X\to X$ which induces a sequence of points $\set{\bm x,f( \bm 
 x),f^2(\bm  x), \cdots}$ in $X$ starting from $\bm x$. A point $\bm  x^*\in X$ is said to be an \textit{equilibrium} or \textit{fixed} point of $f$ if $f(\bm x^*) = \bm x^*$. 

\begin{definition}[Globally Asymptotically Stable (GAS)]\label{def:gas}
An equilibrium point $\bm  x^*$ is said to be \textit{stable in the Lyapunov sense} for $f$ if for every $\epsilon>0$ there is some $\delta_\epsilon>0$ such that $d(\bm x,\bm x^*)<\delta_\epsilon$ implies $d(f^t(\bm x),\bm x^*) < \epsilon$ for each integer $t\ge 0$. An equilibrium point $\bm x^*$ is said to be \emph{globally asymptotically stable} (GAS) if it is Lyapunov stable and $\lim_{t\to\infty} d(f^t(\bm x),\bm x^*) = 0$ for every $\bm x\in X$.
\end{definition}
A special case of the above is a \emph{discrete-time linear dynamical system} (DTLDS), in which we will consider $X \subset \R^n$ and $f$ to be a linear map given by matrix $M \in\R^{n\times n}$ as $\bm x_{t+1} = M \bm x_t$.
It is clear that a GAS system cannot have two equilibrium points in $X$. Also, for any norm $\norm{\cdot}{}$ on $\R^n$ a DTLDS is GAS with $\bm x^*=\bm 0$ \emph{iff} every eigenvalue of $M$ has magnitude $<1$ ($\rho(M) < 1$). We will usually take $\norm{\cdot}{} = \norm{\cdot}{1}$ so that $d(\bm x,\bm y) = \norm{\bm x-\bm y}{1}$, unless stated otherwise.

One example of a GAS DTLDS which is \emph{not on} $\R^n$ but on the subspace $\Delta_{n}$ and is still described by a matrix is the well known concept of an \emph{Markov chain}. This system is described by a column stochastic matrix $M\in\R^{n\times n}$ with non-negative entries and each column summing to $1$, and additional properties for ergodicity~\citep{freedman2012markov}.
It describes a dynamics on the probability simplex $\Delta_{n}$ because if $\bm x\in \Delta_{n}$ then $M\bm x \in \Delta_{n}$. It is known that if $M$ has only one eigenvalue equal to $1$ and other eigenvalues have magnitude strictly less than $1$ then $\exists! \;\bm \pi\in\Delta_n$ such that $M\bm \pi = \bm \pi$. This is the \textit{stationary distribution} of the Markov chain and for any $\bm x\in \overline{\Delta}_n$ (note, not just $\Delta_n$), $\lim_{t\to\infty}\norm{M^t\bm x-\bm \pi}1 = 0$.
In this work, we consider GAS DTLDS systems for the core problem stated in Section~\ref{sec:statement}. Our results apply to Markov chains WLOG, and we next relate GAS systems and Markov chains to support this claim.

\section{GAS System and Markov Chain Relation}\label{sec:relation}

Consider a Markov chain on $\Delta_{n}$ with transition matrix $M$ that has only one eigenvalue equal to $1$ and other eigenvalues have magnitude strictly less than $1$.
Pick $A\in \R^{(n-1)\times n}, B\in \R^{n\times (n-1)}$ with the properties listed below. Here, we treat $A$ and $B$ as a linear operators $A : \R^n \rightarrow \R^{n-1}$ and $B : \R^{n-1} \rightarrow \R^{n}$, and use the shorthand $A(X) = \{A(x)~|~ x \in X\}$ and $B(Y) = \{B(y)~|~ y \in Y\}$. Further, for any operator $F: X \rightarrow Y$, the restriction $F|_Z$ ($Z \subseteq X$) is the operator with restricted domain $F|_Z : Z \rightarrow Y$. The properties desired of $A,B$ are
\begin{itemize}
    \item $B(\R^{n-1}) = H_n$.
    \item $A|_{H_n}$ is an isomorphism onto $\R^{n-1}$.
    \item $AB = I_{n-1}, \left.(BA)\right|_{H_{n}} = \left.I\right|_{H_n}$.
\end{itemize}
Then, $M$ gives a GAS DTLDS on $\R^{n-1}$  determined by $\overline M = AMB \in\R^{(n-1)\times (n-1)}$. This construction can be viewed intuitively as moving between the probability simplex and $H_n$ using the linear operators $A,B$, which is natural but this result has not appeared in literature, as far as we know.
The above properties can yield many possible $A,B$, an example $A,B$ is as follows: $A \in \R^{(n-1)\times n}$ and $B \in \R^{n\times (n-1)} $ such that 
$$A_{ij} = \begin{cases}
1 & \text{for}\ j \leq i \\
0 & \text{otherwise}
\end{cases}
, \quad B=\begin{bmatrix}\bm 0^\top\\ 
-I_{n-1} \end{bmatrix} + \begin{bmatrix} I_{n-1} \\ \bm 0^\top \end{bmatrix} \; .
$$
The above three properties can be verified in a straightforward way for this example $A,B$.
We will next show that for any $\bm x_0 \in \Delta_{n-1}$, the sequence of states $\bm x_1 = M\bm x_0, \bm x_2 = M \bm x_1, \ldots$ converging to $\bm \pi$ is mirrored by the system $\overline M = AMB$. Before that, we show some structural relations.
\begin{theorem}[Structural Relation] \label{thm:structure}
For $A,B,M$ and $\overline{M}$ as defined above, we have
\begin{enumerate}
    \item $B \overline M = M B$.
    \item $\overline M A \bm z=A M \bm z$ for any $\bm z \in H_{n}$.
    \item $\overline M^t = A M^t B$.
\end{enumerate}
\end{theorem}

\begin{proof}{\informsproof}
\begin{subequations}\label{eq:comm-diags}
\noindent
\begin{minipage}[t]{0.43\linewidth}
  \vspace{0pt}
  \begin{equation}\label{diag:B}
  \begin{tikzcd}
    \R^{n-1} \arrow[r,"B"] \arrow[d,"\overline{M}"'] & H_n \arrow[d,"M"] \\
    \R^{n-1} \arrow[r,"B"]                            & H_n
  \end{tikzcd}
  \end{equation}

  \begin{equation}\label{diag:A}
  \begin{tikzcd}
    H_n \arrow[d,"\overline{M}"']  & H_n \arrow[d,"M"'] \arrow[l,"A"'] \\
    \R^{n-1}  & \R^{n-1} \arrow[l,"A"']
  \end{tikzcd}
  \end{equation}

  \begin{equation}\label{diag:commute}
  \begin{tikzcd}
    \R^{n-1} \arrow[r,"B"] \arrow[d,"\overline{M}"'] & H_n \arrow[d,"M"] \\
    \R^{n-1}                          & H_n \arrow[l,"A"'] 
  \end{tikzcd}
  \end{equation}
\end{minipage}\hfill
\begin{minipage}[t]{0.44\linewidth}
  \vspace{0pt}
  \begin{equation}\label{diag:stack}
  \begin{tikzcd}
    \R^{n-1} \arrow[r,"B"] \arrow[d,"\overline{M}"'] & H_n \arrow[d,"M"] \\
    \R^{n-1} \arrow[r,"B"] \arrow[d,"\overline{M}"'] & H_n \arrow[d,"M"] \\
    \vdots     \arrow[d,"\overline{M}"']             & \vdots \arrow[d,"M"] \\
    \R^{n-1} \arrow[r,"B"] \arrow[d,"\overline{M}"'] & H_n \arrow[d,"M"] \\
    \R^{n-1}                        & H_n \arrow[l,"A"']   
  \end{tikzcd}
  \end{equation}
\end{minipage}
\end{subequations}

We prove each part below using the commuting diagrams above:
\begin{enumerate}
\item For any $\bm v \in \R^{n-1}, B\bm v \in H_{n} \implies MB\bm v\in H_n$. But $BA$ acts as identity on $H_n$ whence $B\overline M \bm v= BAMB\bm v = BA(MB\bm v) = MB\bm v$. Since this is true for all vectors in the domain of $B\overline M$ (and $MB$), these two operators must be equal. This proves that the diagram in (\ref{diag:B}) commutes.
\item Say $\bm z\in H_n$. Then $BA\bm z = \bm z$ because $\left.(BA)\right|_{H_n} = I|_{H_n}$. So $\overline MA\bm z = AMBA\bm z = AM\bm z$. This proves that the diagram in (\ref{diag:A}) commutes.
\item Using $B \overline M  =  MB$ from property 1, multiply both sides by $\overline M$ to obtain $B \overline M^2 = M B \overline M = M M B = M^2 B$. Repeating this,  it follows that $B\overline M^t = M^t B $. Then, as $AB = I_{n-1}$ hence $AM^tB = \overline M^t$. 
\end{enumerate}
Alternately, for the last property 3, here's a diagrammatic proof : take any $\bm y\in\R^{n-1}$ and note that by definition of $B$, we have $B \bm y \in H_{n}$.  Obtain $\bm x = B \bm y + \bm \pi$ so that $\bm x\in \overline \Delta_{n-1}$ and $\bm x - \bm \pi \in H_{n}$. Then, $M \bm x = M B \bm y + \bm \pi$ since $M \bm \pi = \bm \pi$. Next, $A(M \bm x - \bm \pi) = AMB \bm y = \overline M \bm y$.
These equations can be represented by the commutative diagram in (\ref{diag:commute}).
By stacking $t-1$ copies of (\ref{diag:B}) and one copy of (\ref{diag:commute}), as shown in (\ref{diag:stack}), we obtain the desired property 3.
\informsendofproof \end{proof}


In particular, the following corollary follows from the second property in the theorem above.
\begin{corollary} \label{cor:structure}
    Let $\bm v_i \sett A(\bm x_i - \bm \pi)$ for the sequence of states $\bm x_0, \bm x_1, \ldots$ generated by $M$. Then, $\bm v _1 = \overline M \bm v_0, \bm v_2 =  \overline M \bm v_1 \ldots$ with the sequence converging to $\bm 0$. We can recover $\bm x_i$ from the system $\overline M$ using $\bm x_i = B \bm v_i + \bm \pi$
\end{corollary}

This result above implies that we can obtain all the states in the dynamics of $M$ starting from $\bm x_0$ by running the system $\overline M$ starting from $A(\bm x_0 - \bm \pi)$. A result similar to the above corollary but in the opposite direction (obtain states for $\overline M$ by running $M$) can be obtained by leveraging the first property in Theorem~\ref{thm:structure}.
Finally, using the third property in Theorem~\ref{thm:structure}, we prove the following.

\begin{proposition} \label{prop:GAS}
$\overline M$ is a GAS system on $\R^{n-1}$ with fixed point $\bm 0$, hence all eigenvalues of $\overline M$ have magnitude $<1$, i.e., $\rho(\overline M) < 1$.
\end{proposition}




\section{Model and Problem Statement}\label{sec:statement}
Given a DTLDS with dynamics $\bm x_{t+1} = M \bm x_t$ starting at a fixed known $\bm x_0$, we define a cost incurred when the system stops at a state $\bm x$.
The cost $g(\bm x) = \langle \bm c, \bm x \rangle
$ is defined as a linear function in $\bm x$ given by constant vector $\bm c$. This is a natural cost for Markov chains, where $\bm x$ is a probability distribution over underlying states and expected cost is $\bm x$-weighted average of the cost of each state. Following our analysis in the previous section, we can WLOG restrict our attention to GAS DTLDS only. This claim is supported by Corollary~\ref{cor:structure} whence for any Markov chain $M$, the cost $\langle \bm c, \bm x \rangle = \langle B^T \bm c, \bm v \rangle + \langle \bm c, \bm \pi\rangle$ where $\bm v$ is a state of the corresponding GAS $\overline M$ as previously defined; cost is again linear in $\bm v$. Thus, WLOG we let $M$ be a GAS DTLDS and continue with using the notation $\bm x$ for states.

The uncertainty in when the system stops is described by a probability distribution $\bm p = \{p_t\}_{t \in \mathcal{T}}$ over a set of consecutive times steps $\mathcal{T}$ ($|\mathcal{T}|$ can be infinity); then, the overall expected reward is $\sum_{t \in \mathcal{T}} p_t \langle \bm c, \bm x_t \rangle $, where $\bm x_t = M^t \bm x_0$. Based on our motivation, we do not know the true $\bm p$ but observe finitely many samples from the distribution. Given the samples, we can form an nominal distribution $\hat{ \bm p} = \{\hat{p}_t \}_{t \in \mathcal{T}}$, which could be the empirical distribution or any distribution estimated from the samples. Then, we consider an ambiguity set of distributions 
\begin{align}\cP= \set{ \bm q \in \Delta(\cT) \st W_1(\hat{ \bm p}, \bm q) \leq \xi } \; \label{eq:ambiguity},\end{align}
where $W_1$ is a Wasserstein-1 distance, which is a popular choice for ambiguity set in distributional robustness.
In order to be robust, we consider the worst case cost over this ambiguity set, specified as 
the Distributional Robust Cost Estimation (\ref{eq:dre}) problem below. This takes a robust estimation form when $\xi$ is large enough {\color{black}and $\cP$ is the entire probability simplex}. In such a case, all dirac-delta distributions over $\mathcal{T}$ are allowed and then due to the linear nature of the objective, the problem becomes Robust Cost Estimation~\eqref{eq:re} below.
We aim to solve the problems~\ref{eq:dre} and~\ref{eq:re}.
\begin{align}
\max_{\bm q \in \mathcal{P}} \; \sum_{t \in \mathcal{T}}  q_t \inner { \bm c} {\bm x_t } \tag{\sf{DRCE}} \label{eq:dre} \\
\max_{t \in \mathcal{T}} \; \langle  \bm c , \bm x_t \rangle \tag{\sf{RCE}}  \label{eq:re}
\end{align}

\emph{Remark}: While our focus here is on temporal robustness, we can also handle additional uncertainty in the initial state $\bm x_0$. More formally, let this uncertainty be specified by a set $\ds\cU = \set{\bm x \st \bm x = \hat{\bm x}_0 + \sum_{i=1}^U \alpha_i \bm u_i, \sum_{i=1}^U \alpha_i = 1, \alpha_i \geq 0 ~ \forall i}$ for given vertices $\lp\bm u_i\rp_{i \in [U]}$  and a nominal $\hat{\bm x}_0$; such vertices can arise in distributional ambiguity set, as we show in the sequel for Wasserstein polytope. The result below reduces the overall problem $\ds\max_{\bm x_0 \in \mathcal{U}, \bm q \in \mathcal{P}} \;  \sum_{t \in \mathcal{T}}  q_t \inner { \bm c} {\bm x_t }$ to solving~\ref{eq:dre} $U$ times. A similar result with~\ref{eq:re} is a simple corollary of this result.

\begin{proposition} \label{prop:initial}
$\ds\max_{\bm x_0 \in \mathcal{U}, \bm q \in \mathcal{P}} \;  \sum_{t \in \mathcal{T}}  q_t \inner { \bm c} {\bm x_t } = \max_{i \in [U]} \max_{\bm q \in \mathcal{P}}  \sum_{t \in \mathcal{T}} q_t \langle \bm c, M^t (\hat{\bm x}_0 + \bm u_i )\rangle $.
\end{proposition}


\section{Approach}\label{sec:approach}
We describe our approach for finite and infinite support of $\hat{\bm p}$ in the two subsections next. Finite support could arise when $\hat{\bm p}$ is taken as the empirical distribution and infinite support when $\hat{\bm p}$ is an estimated parametric distribution, given data samples. WLOG and to avoid cumbersome notation, we assume that $\mathcal{T} = \set{1,2,\cdots,T} = [T]$, where $T$ could be infinite. 

\subsection{Finite Support: Robust Cost Estimation} \label{sec:finitesupportrce}
\textit{Naive approach}: The naive way of solving the \ref{eq:re} problem is to evaluate $\inner{\bm c}{M^t \bm x_0}$ for each $t \in \mathcal{T}$. It is known that the dot product of two vectors (of size $\cO(n)$) has $\cO(n)$ time complexity, and the standard approach of $n$ dot products to compute a matrix-vector product has $\cO(n^2)$ time complexity. Using the fact that $M^{t} \bm x_0 = M \times M^{t-1} \bm x_0$, a naive approach would require $T$ matrix-vector and $T$ dot products, yielding a time complexity of $\cO(n^2T)$.

\textit{Prior results based approach}: A slightly better approach can be obtained by observing that $M^t \bm x_0 = M \times M^{t-1} \bm x_0$ involves multiplying the same matrix $M$ with different vectors. A prior result~\cite{williams2007matrix} has shown that by preprocessing the matrix $M$, the time complexity of $T$ matrix-vector products is $\cO(n^{2 + \epsilon} + T n^2/(\epsilon \log n)^2)$ for any small positive constant $\epsilon \in \lp0,\frac12\rp$.  Thus, if $T$ is larger than $n$, then the preprocessing step of $\cO(n^{2 + \epsilon})$ can be beneficial.

\textit{Our approach}: We take this approach further in an algorithm that we call Small and Big Strides (SaBS) shown in Algorithm~\ref{alg:robustfinite}. We start with the observation that $\langle c, M^t \bm x_0 \rangle = \langle M^\top c, M^{t-1} \bm x_0 \rangle$, and applying this recursively $\langle c, M^t \bm x_0 \rangle = \langle (M^\top)^j c, M^{t-j} \bm x_0 \rangle$. Let $B = \lfloor\sqrt{T}\rfloor$. First, we compute $M^B$ (line 2) by successive squaring $M^1, M^2, M^4$, a popular approach in number theory~\cite{karp1984exponential}. Then, preprocess $M^B, M^\top$ and $M$ for later use in multiple matrix-vector products. Then, precompute $M^{B } \bm x_0$, $M^{2 B} \bm x_0$, $ \ldots, M^{B^2} \bm x_0$ and $(M^\top)^1 \bm c, \ldots, (M^\top)^{B-1} \bm c$. Then, using the fact that ${\color{black}\langle c, M^t \bm x_0 \rangle = \langle (M^\top)^{t \mod B} c, M^{\lfloor t/B \rfloor B} \bm x_0 \rangle}$, we can compute all $\langle c, M^t \bm x_0 \rangle$ for each $t \in \{1, \ldots , B^2\}$ using the precomputed values in the loop at line 6. The remaining powers from $B^2+1$ to $T$ is computed in loop starting at line 7. 
We achieve a better runtime dependence on $T$:
\begin{lemma} \label{lem:robustruntime}
$\fa\epsilon\in\lp0,\frac12\rp$, Algorithm \ref{alg:robustfinite} can be performed in $\cO\left(n^\omega \log T + n^{2 + \epsilon} + \sqrt{T} n^2/(\epsilon \log n)^2 + nT \right)$ time, where $n^\omega$ is the time complexity of matrix multiplication. (best known $\omega$ currently is $\approx 2.3715$~\citep{williams2024new}). 
\end{lemma}
\begin{proof}{\informsproof}
The complexity of finding $M^{B}$ (line 1) is $\cO(n^\omega \log B )$ as there are about $\log B$ squares and multiplications in the successive squaring (matrix products). Precomputing $B$ terms in line 3 takes  $\cO(n^{2 + \epsilon} + B n^2/(\epsilon \log n)^2)$ time. Similarly, precomputing the matrix transpose powers in line 4 takes $\cO(n^{2 + \epsilon} + B n^2/(\epsilon \log n)^2)$ time. The first for loop (line 5) has time complexity of $\cO(n T)$ because of one dot product in the loop that runs $B^2$ times. The second loop (line 7) does a matrix-vector product with the same matrix $M$ a max of $\sqrt{T}$ times, which has time complexity $\cO \left(n^{2 + \epsilon} + \sqrt{T} n^2/(\epsilon \log n)^2\right)$.
\informsendofproof \end{proof}

\begin{algorithm}[t]
\caption{Small and Big Strides (SaBS)} \label{alg:robustfinite}  
\DontPrintSemicolon
\SetAlgoLined
Let $B = \lfloor \sqrt{T} \rfloor$. Compute $M^{B}$ by successive squaring (see text).\; Preprocess $M^{B}, M^\top, M$ as in~\cite{williams2007matrix}.\;
 Let $\bm v^j = M^j \bm x_0$. Compute $M^{B} \bm x_0, M^{2B} \bm x_0$, $\ldots$, $M^{B^2} \bm x_0$ using $\bm v^{(i+1)B} = M^{B} \bm v^{iB} $.\;
 \tcc{Product of $M^B$ with many vectors as in~\cite{williams2007matrix} }
Compute $(M^\top)^1 \bm c$, $(M^\top)^2 \bm c$, $\ldots$, $(M^\top)^{B-1} \bm c$ using $(M^\top)^{i+1} \bm c = M^\top (M^\top)^{i} \bm c$.\;
 \tcc{Product of $M^\top$ with many vectors as in~\cite{williams2007matrix} }
\For{each $t \in \{1, \ldots, B^2\}$}{
    Obtain $\langle \bm c, M^t \bm x_0 \rangle =\langle (M^\top)^{ t\mod B} \bm c, M^{ \lfloor t / B\rfloor B} \bm x_0 \rangle$ using precomputed values above.\;
}
\For{each $t \in \{B^2+1, \ldots, T\}$}{ 
    Compute $M^{t} \bm x_0  = M \bm v^{t-1}$, and then $\langle \bm c, M^t \bm x_0 \rangle $. Let $\bm v^{t} = M^{t} \bm x_0$.  \;
     \tcc{Product of $M$ with many vectors~\citep{williams2007matrix} }
}
\textbf{Output:} All $\langle \bm c, M^t \bm x_0 \rangle$ for $t \in \mathcal{T}$.
\end{algorithm}


The result above improves the dominating term $T n^2/(\epsilon \log n)^2$ from the previous approach to $\sqrt{T} n^2/(\epsilon \log n)^2$ and introduces the terms $nT$ that is better than $T n^2/(\epsilon \log n)^2$. Also, for $T$ larger than $n$, the term $n^{\omega}\log T$ is better than $T n^2/(\epsilon \log n)^2$. 

As a final remark, another plausible approach is to find the real Jordan form of $M=PJP^{-1}$ and use the fact that $M^k = PJ^k P^{-1}$. Given $J$, $J^k$ has a closed form formula in terms of entries of $J$.
However, computing the Jordan form itself takes $\cO(n^3)$ time, and then after this one would still need to compute $M^t \bm x_0$ for each $t$ which is $\cO(n^2 T)$, which provides an overall worse time complexity than our result. Note that we cannot employ the prior result~\cite{williams2007matrix} for this plausible approach since in this approach, we do not use the same matrix in multiple matrix-vector products.




\subsection{Finite Support: Distributional Robust Cost Estimation} Wasserstein ambiguity sets have been applied in optimization and machine learning~\citep{KuhnWasserstein, EsfahaniWasserstein}.
First, we present a useful and general geometric characterization of Wasserstein distance in finite dimensions that has appeared in recent literature~\citep{indep2,indep,wass} without a detailed proof, thus, we provide a detailed and intuitive proof (Theorem~\ref{thm:polytope}) of the characterization that relies on a separating hyperplane argument. Fix a distance $d:\cT\times \cT\to\R_{\ge 0}$. This is essentially given by a symmetric matrix with entries $d_{ij}$ such that $d_{ii}=0\fa i\in\cT, d_{ij}>0$ whenever $i\ne j\in\cT$, and $d_{ij}+d_{jk} \ge d_{ik}\fa i,j,k\in\cT$. The Wasserstein distance $W^d(\bm \rho,\bm\nu)=W_1^d(\bm \rho,\bm\nu)$ based on the metric $(d_{ij})$ between two probability distributions $\bm \rho,\bm \nu\in \Delta(\cT)$ is defined for finite $\cT$ in a dual form as 
$$\max \Big\{\sum_{i=1}^T(\rho_i-\nu_i)x_i\st \abs{x_i-x_j}\le d_{ij}\fa i,j\in\cT \Big\}.$$ 
The above definition is a direct result of Kantorovich dual formulation of Wasserstein distance~\citep{villani2008optimal}.
Recall that any norm $\norm{\cdot}{}$ in a vector space $X$ is a function that satisfies certain properties for any $x, y \in X$: (1) \emph{triangle inequality}: $\norm{x+y}{} \leq \norm{x}{} + \norm{y}{}$,
(2) \emph{absolute homogeneity}: $\norm{ax}{} = |a|\norm{x}{}$, and
(3) \emph{positivity}: $\norm{x}{} = 0$ iff $x = 0$.
Recall $H_T$, which was defined as $\{\bm x\in \R^{T} \st \bm 1^\top \bm x=0 \}$, is a subspace of $\R^T$. Then, the Wasserstein distance $W^d_1$ induces a norm on $H_T$.
\begin{proposition}\label{prop:wasserstein-norm} $\norm{\cdot }{W,d}$ as specified below is a norm on $H_T$; for any $\bm \mu \in H_T$
$$\norm{\bm \mu}{W,d} \sett \max\big\{\bm \mu^\top\bm x\st \bm x\in\R^T, \abs{x_i-x_j}\le d_{ij}\fa i,j \big\}.$$ 
\end{proposition}
Observe that the constraint set of the above linear program is unbounded, that is, if $\bm x$ satisfies the constraints, then so does $\bm x+\lambda\bm 1$ for any real number $\lambda$ and also, given $\bm\mu\in H_T$, $\bm \mu^\top(\bm x+\lambda\bm 1) = \bm \mu^\top\bm x$. This observation allows for an equivalent formulation further restricting $\bm x$ to a bounded set
$$\norm{\bm \mu}{W,d} = \max\big\{\bm \mu^\top\bm x\st \bm x\in H_T, \abs{x_i-x_j}\le d_{ij}\fa i,j \big\} .$$ 

\iftoggle{isOR}{}{\begin{figure}
\centering
\includegraphics[scale=0.15]{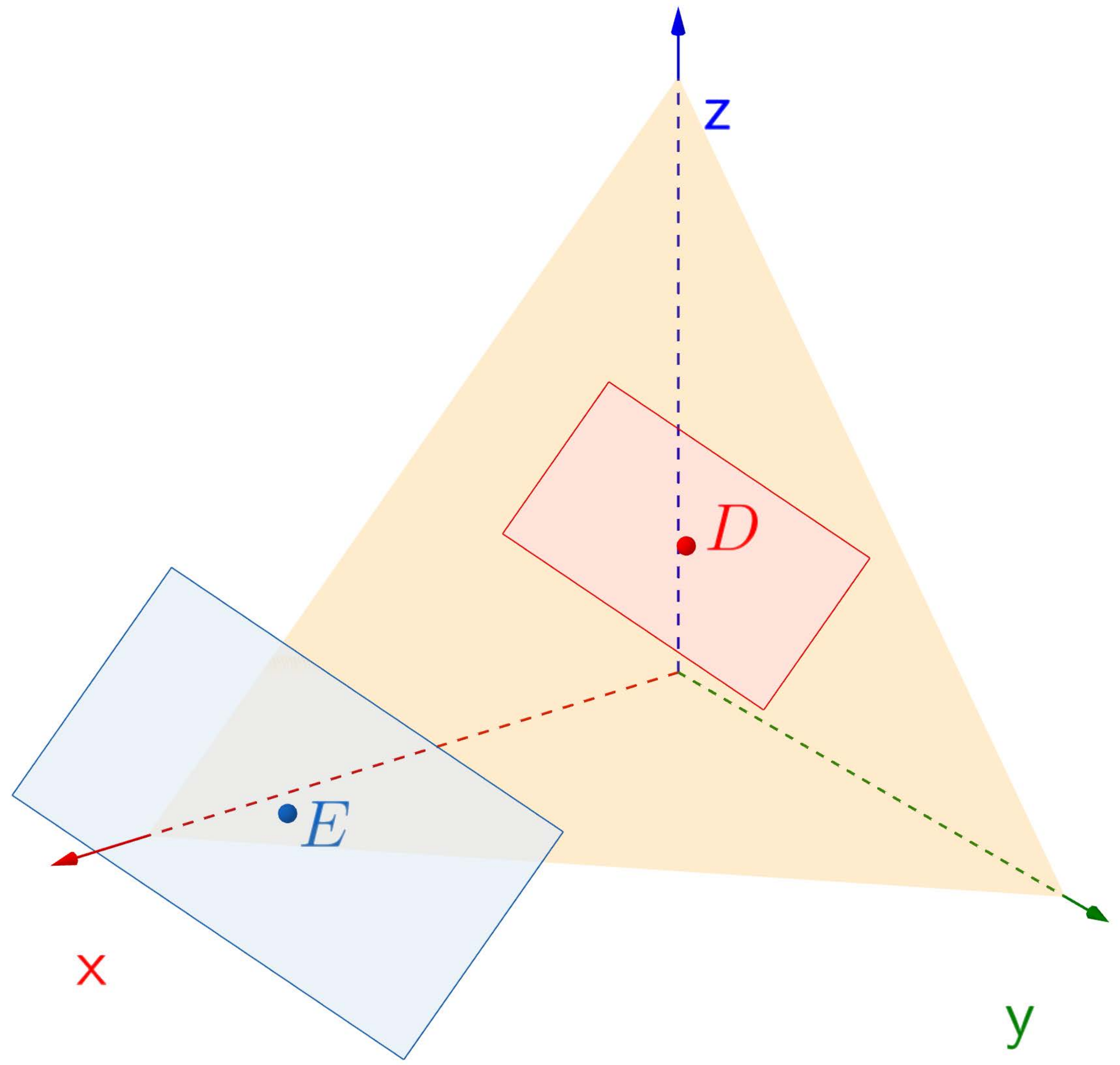}
\caption{$\norm{\cdot}{W}^d$ ($d_{ij}=\abs{i-j}$) balls on the probability simplex in 3d space with {\color{blue}$\xi=0.3$ centered at $E=(0.83,0.13,0.04)$} and {\color{red}$\xi=0.2$ centered at $D=(0.24,0.35,0.41)$}.} \label{fig:normball}
\end{figure}}

Next, recall that our ambiguity set $\cP$ in Equation~\ref{eq:ambiguity} contain $\bm q \in \Delta_T$ such that $W^d_1(\hat{\bm p}, \bm q) \leq \epsilon$. From the definition of $W^d_1$ above and the norm, we can see that this is the same as $\norm{\bm\mu}{W,d} \leq \xi$ for $\bm\mu = \bm q - \hat{\bm p}$. Further, since $q_i \geq 0$, this leads to following conclusion in a straightforward manner:
$$
\cP  = \{ \bm q \st \bm q = \hat{\bm p} + \bm\mu, \bm \mu \in H_T, \norm{\bm\mu}{W}^d \leq \xi, \bm q \geq 0\}
$$
Alternatively, the above can be thought of as the intersection of the probability simplex $\Delta_T$ and the norm ball $\norm{\bm\mu}{W}^d \leq \xi$ shifted by $\hat{\bm p}$. An illustration in $\R^3$ (i.e., $\abs T=3$) is shown in Figure~\ref{fig:normball}.
Thus, we investigate the structure of the norm $\norm{\cdot}{W,d}.$ We show that 
the unit ball of this norm is a $T-1$ dimensional polytope. 

\begin{theorem} \label{thm:polytope}
The Wasserstein unit ball in $H_T$ around the origin $\bm 0\in\R^T$ based on the distance $d:\cT^2\to\R_{\ge 0}$ is the convex hull of the points $\left\{\frac{\bm e_i-\bm e_j}{d_{ij}}\st i\ne j\right\}$.
\end{theorem}
\begin{proof}{\informsproof}
For the Wasserstein ball of radius $1$ around $\bm 0\in\R^T$, we want to find those $\bm \mu\in H_T$ such that $\bm \mu^\top\bm x \le 1$ for every $\bm x\in H_T$ satisfying $A\bm x\le \bm 1, \bm 1^\top\bm x\le0, -\bm 1^\top \bm x\le 0$. Here $A$ is the matrix whose rows comprise of the constraint vectors $\frac{\bm e_i-\bm e_j}{d_{ij}}$, hence $A$ has $m = T(T-1)$ rows. Also, $\bm 1^\top\bm x\le0, -\bm 1^\top \bm x\le 0$ is same as $\bm 1^\top\bm x = 0$, which enforces $\bm x \in H_T$.

Take $W = \begin{bmatrix}
A\\\bm 1^\top\\-\bm 1^\top
\end{bmatrix}\in\R^{(m+2)\times T}$ and $\bm v = \begin{bmatrix}
    \bm 1\\0\\0
\end{bmatrix}\in \R^{m+2}$. Then, $A\bm x\le \bm 1, \bm 1^\top\bm x\le0, -\bm 1^\top \bm x\le 0$ can be written as $W\bm x \le \bm v$.

Fix such a $\bm \mu$ in the Wasserstein unit ball around $\bm 0$. Then $S \sett \set{\bm x\in\R^{T}\st W\bm x \le \bm v, \bm\mu^\top\bm x -1>0}$ must be empty. Alternately $S=\set{\bm x\in\R^{T}\st \begin{bmatrix}W&-\bm v\end{bmatrix}\begin{bmatrix} \bm x\\1\end{bmatrix} \le \bm 0, 
\begin{bmatrix}\bm\mu^\top&-1\end{bmatrix} \begin{bmatrix}\bm x\\1\end{bmatrix}>0} = \varnothing$. Define $B \sett \begin{bmatrix}
W^\top\\-\bm v^\top
\end{bmatrix}, \bm b\sett \begin{bmatrix}
    \bm \mu\\-1
\end{bmatrix}$, and homogenize $S$ to $S' \sett \set{\bm u\in\R^{T+1}\st B^\top\bm u \le \bm 0, \bm b^\top\bm u>0}$. Note that the last component of vector $\bm u$ in $S'$ is a free variable as opposed to the last component in $S$. So while $S$ is empty, we need a proof for the claim below:
\paragraph{Claim.}
 $S'=\varnothing$.

\begin{proof}{\informsproof}
We always denote $\bm u=\begin{bmatrix}\bm x\\y\end{bmatrix}$ with $\bm x\in\R^T, y\in\R$. Recall that $B^\top=\begin{bmatrix}W&-\bm v\end{bmatrix} = \begin{bmatrix}A&-\bm 1\\\bm 1^\top&0\\-\bm 1^\top&0\end{bmatrix}$. Suppose $\bm u\in S'$. Then the first $m$ rows of $B^\top \bm u\le 0$ gives the condition that $A\bm x \le y$, i.e., $x_i-x_j\le y\fa i\ne j$. Adding the inequalities for, say $(i,j)=(1,2),(2,1)$, gives $y\ge 0$. Consider the following two cases.
\begin{itemize}
\item $y=0$: The first $m$ rows give the condition $x_i-x_j\le 0\fa i\ne j$ and the last two rows give that $\bm 1^\top \bm x=0$. The former implies $x_i=x_j\fa i\ne j$. Combined with $\bm 1^\top \bm x=0$, we get $\bm x=\bm 0$. So $\bm u=0$ is the only candidate with last coordinate $0$ that satisfies $B^\top \bm u\le \bm 0$. But this does not satisfy $\bm b^\top\bm u>0$. So $S'$ has no element with last coordinate $0$.
\item $y>0$: Then the vector $\bm v = \frac{1}{y}\bm u = \begin{bmatrix}\frac1y\bm x\\1\end{bmatrix}\in S=\varnothing$. This is impossible. 
\end{itemize}
\informsendofproof \end{proof}

Then, $\set{\bm \alpha=\begin{bmatrix}
\bm \lambda\\r\\s
\end{bmatrix}\in \R^{m+2}\st B\bm \alpha =\bm b,\bm \alpha\ge 0}\ne\varnothing$ by Farkas lemma~\ref{lem:farkas}. Pick one such element $\begin{bmatrix}
\bm \lambda\\r\\s
\end{bmatrix}$ of this set. Recall $B = \begin{bmatrix}
A^\top&\bm 1&-\bm 1\\
-\bm 1^\top&0&0
\end{bmatrix}\in \R^{(T+1)\times (m+2)}$, $\bm b=\begin{bmatrix}
    \bm\mu\\-1
\end{bmatrix}\in\R^{T+1}$. The last row constraint gives $-\bm 1^\top\bm \lambda=-1$ so that $\bm 1^\top \bm \lambda =1$. The first $T$ row constraints give $A^\top\bm\lambda +(r-s)\bm 1 = \bm \mu$. But $\bm\mu\in H_T$. So $\bm 1^\top\bm \mu=0\implies (A\bm 1)^\top+(r-s)\bm1^\top\bm 1 = (r-s)T\implies r=s\implies \bm\mu=A^\top\bm \lambda$. Here we used the fact that $A\bm 1=\bm 0_{m}$ because each row of $A$ sums to $0$. A typical column of $A^\top$ looks like $\bm v_{ij}=\dfrac{\bm e_i-\bm e_j}{d_{ij}}$. This establishes that $\exists\bm\lambda \in (\R_{\ge0})^{m}$ indexed by $(i,j),i\ne j$, such that $\sum\limits_{j\ne i}\lambda_{ij}=1$ and $\bm \mu = \sum\limits_{j\ne i}\lambda_{ij}\bm v_{ij}$. It can be easily checked that the $\bm v_{ij}$'s satisfy our required criterion of the Wasserstein distance from $\bm 0$ being $1$. 
\informsendofproof \end{proof}
Following from the general result above, 
for our problem $d_{ij}$ represent distance between time points on the line. It is natural to thus consider the distance function to be given by $d_{ij}=\abs{i-j}.$ Then, the set of corners proposed in the above theorem are $T(T-1)$ in number. However, a further observation with the mentioned distance function yields that if $s<t\in[T]$ then $$\bm v_{st}=\frac{\bm e_t-\bm e_s}{d_{st}} = \frac{1}{t-s}\sum_{k=s}^{t-1}\left(\bm e_{k+1}-\bm e_k\right) = \frac{1}{t-s}\sum_{k=s}^{t-1}\bm v_{k+1,k}$$ and similar if $s>t$. This implies that $\bm v_{st}$ for $|s - t| > 1$ cannot be an extreme point of the polytope as it is a convex combination of other extreme points.  Thus, there are only $2(T-1)$ extreme points of the Wasserstein unit ball polytope, namely $\pm\bm v_{1,2},\cdots, \pm\bm v_{T-1,T}$. For the $\xi$  Wasserstein ball polytope, these extreme points will be just multiplied by $\xi$. As a consequence, using $\text{conv}$ to denote convex hull, we can write
\begin{align*}
\cP  = \big \{ \bm q \st \bm q = \hat{\bm p} + \bm\mu, \bm \mu \in  \text{conv}(\pm \xi \bm  v_{1,2},\cdots, \pm \xi \bm v_{T-1,T}), \bm q \geq 0 \big \}.
\end{align*}
Recall that~\ref{eq:dre} is a linear program (LP) with the constraint as $\cP$. Then, there are two cases.

\textbf{Case 1}: $\hat{\bm p} + \bm\mu \in \Delta_T \fa \bm \mu \in \text{conv}(\pm\xi\bm v_{1,2},\cdots, \pm\xi\bm v_{T-1,T})$. This case is illustrated by the polytope with center $D$ in Figure~\ref{fig:normball}. An LP has optimum at one of the extreme points of its constraint set.
Hence, while solving an LP on this polytope, one needs to compute the objective value of the program only at the $\cO(T)$ extreme points. In other words, when optimizing over $\cP$, where $\xi$ is small enough to make each extreme points $\hat{\bm p} \pm \xi(\bm e_{i}-\bm e_{i+1})\in \Delta(T)$ for $i \in [T-1]$, we only have to evaluate the objective at these extreme points and choose the maximum of the evaluated values. Noting the similarity of this to the robust estimation case, and from Lemma~\ref{lem:robustruntime}, it is straightforward to check that the runtime is $\cO(n^\omega \log T + n^{2 + \epsilon} + \sqrt{T} n^2/(\epsilon \log n)^2 + nT )$.

\textbf{Case 2}: For some $\bm \mu \in \text{conv}(\pm\xi\bm v_{1,2},\cdots, \pm\xi\bm v_{T-1,T})$, $\hat{\bm p} + \bm\mu \notin \Delta_T$. This case is illustrated by the polytope with center $E$ in Figure~\ref{fig:normball}. In this case the number of extreme points can be large, but we can write this problem as an LP as follows:
\begin{align}
\max_{\bm q \in \R^T, \lambda_i \in \R} & \; \sum_{t \in \mathcal{T}}  q_t \inner { \bm c} {\bm x_t } \qquad \text{s.t.} \qquad  \bm q \geq 0,\; \sum_i \lambda_i = 1, \; \lambda_i \geq 0 \;  \forall i\;, \label{eq:LP}\\
 & \;\; \bm q = \hat{\bm p} + \lambda_1 \xi \bm v_{1,2} + \ldots + \lambda_{T-1} \xi \bm  v_{T-1,T} +  \lambda_{T} \bm (-\xi \bm  v_{1,2}) + \ldots  \lambda_{2(T-1)} (-\xi \bm v_{T-1,T}) .\nonumber
\end{align}
The above LP has $\cO(T)$ number of optimization variables, and is represented in $L$ bits, where $L$ is $\cO(T)$ assuming a fixed number of bits for each constant and variable. Then, an LP requires $\cO(n^{3.5} L^2)$ time for solving exactly, which can be improved to $\cO(n^{3.5} L)$ runtime if we allow for any arbitrary approximation due to~\cite{karmarkar1984new}. Note that we also need to compute the $\langle \bm c, \bm x_t \rangle$ for all $t \in \cT$, which is the same as robust estimation case. Thus, considering the dominating terms, we obtain an overall complexity of $\cO(n^{3.5} T^2)$ in the exact case and $\cO(n^{3.5} T)$ with arbitrary approximation.



\subsection{Infinite Support: Robust Cost Estimation}
With an infinite $\cT$, the methods developed till now do not apply as they require computing over $|\cT|$ steps. Thus, we explore a different set of methods for the infinite case. 
Assume $M=PJP^{-1}$ where $J$ is in its reduced real Jordan form in Equation~\ref{jordan}. WLOG let these blocks be ordered such that $r_1<\cdots<r_q$ and $\abs{\lambda_1}<\cdots<\abs{\lambda_p}$.
We also know that all these magnitudes are strictly less than one ($r_i, |\lambda_j|<1$) as our problem statement handles GAS systems. As stated after Equation~\ref{jordan}, $J^t = \text{diag}(r_1^t J_1^t, \ldots, \lambda_p^t)$
where the $i^{\text{th}}$ $2\times 2$ block $J_i^t = r_i^t\begin{bmatrix}\cos t\theta_i&-\sin t\theta_i\\\sin t\theta_i&\cos t\theta_i\end{bmatrix}$.
We consider $\inner{\bm c}{PJ^tP^{-1}\bm x}$ as a function of $t$ which is $\sum_{i=1}^qr_i^t(u_i\cos t\theta_i-v_i\sin t\theta_i) +\sum_{j=1}^p w_j\lambda_j^t$ where $u_i,v_i,w_j$ only depend on $\bm c,\bm x, P$ (please see the details of the standard but tedious steps in Appendix~\ref{sec:misc}). Let $(u_i,v_i) =  (d_i \cos \eta_i, d_i \sin\eta_i)$ where $\eta_i\in[0,2\pi)$ and $d_i=\sqrt{u_i^2+v_i^2}$. So the expression we consider is 
\begin{equation} \label{eq:gt}
    g(t) = \inner{\bm c}{PJ^tP^{-1}\bm x} = \sum_{i=1}^q d_ir_i^t\cos(t\theta_i+\eta_i) + \sum_{j=1}^p w_j\lambda_j^t. 
\end{equation}
Note that the~\ref{eq:re} problem here is to solve $\ds\max_{t \in \N} g(t)$. 
\paragraph{Main Approach.} We first start by finding a $t_0$ such that $g(t_0) > 0$. Then, since our problem is for a GAS system with $r_i, |\lambda_j|  < 1$, we know that $g(t)$ converges to $0$ as $t \to \infty$. Thus, from basic real analysis, there must exists a $n_0$ such that $\forall t' > n_0$, $|g(t')| < g(t_0) $. Our approach is to compute $g(t)$ for $t = 1$ to $n_0$ and choose the maximum among these. Next, we show a construction of $n_0$, given $t_0$.

\begin{proposition} \label{prop:n0}
Given $t_0$, for any positive integer $n_0$ more than $ \Big\lceil\log_{\zeta}{\frac{g(t_0)}{\sum\abs{d_i}+\sum \abs {w_j}}}\Big\rceil$, where $\zeta \sett \max\set{\abs{\lambda_p},r_q}$, we have $\forall t' > n_0, |g(t')| < g(t_0)$.
\end{proposition}
\begin{proof}{\informsproof}
$\zeta = \max\set{\abs{\lambda_p},r_q} < 1$. For $t' > n_0 \geq \log_{\zeta}{\frac{g(t_0)}{\sum\abs{d_i}+\sum \abs {w_j}}}$, since $\zeta < 1$, we have $\ds\zeta^{t'} < \frac{g(t_0)}{\sum\abs{d_i}+\sum \abs {w_j}}$ whence $\ds\abs{g(t')} \le \sum \abs{d_i}r_i^{t'} +\sum \abs{w_j\lambda_j^{t'}} \le  \left(\sum\abs{d_i}+\sum \abs {w_j}\right) \zeta^{t'} < g(t_0)$.
\informsendofproof \end{proof}

The correctness of the approach is not hard to see. Since we already found $t_0$ with $g(t_0) > 0$ and $g(t') < g(t_0) $ for $t' > n_0$ (note that $n_0 \geq t_0$ is implied), then any time step $t^*$ for which $g(t^*)$ is maximum it must be such that $1 \leq t^* \leq n_0$. Also, there is a special case when there is no such $t_0$ that can be found, and then $g$ takes the maximum value of $0$ at $\infty$. Next, we explain how to find such a $t_0$ or infer that there is no such $t_0$. We consider two cases, based on which one among $r_q$ or $|\lambda_p|$ is dominating for large $t$ in function $g$.

\textbf{Case 1}: $r_q<\abs{\lambda_p}$. Then $w_p\lambda_p^t$  is the dominating term in $g$. Dividing $g(t)$ by $\abs{\lambda_p}^t$ we obtain
$$\frac{g(t)}{\abs{\lambda_p}^t} = \sum_{i=1}^{q}d_i\big(\frac{r_i}{\abs{\lambda_p}}\Big)^t\cos(t\theta_i+\eta_i)+ 
\sum_{j=1}^{p-1} w_j\Big(\frac{\lambda_j}{\abs{\lambda_p}}\Big)^t + w_p(\sgn \lambda_p)^t.$$
We show that the LHS $\frac{g(t)}{\abs{\lambda_p}^t}$ gets quite close to the dominating term $w_p\lp\sgn \lambda_p\rp^t$ for large $t$; using triangle inequality we quantify an estimate of the error $\abs{\frac{g(t)}{\abs{\lambda_p}^t}-w_p(\sgn \lambda_p)^t} $ as follows:
\begin{proposition}\label{prop:tightness}
$\abs{\frac{g(t)}{\abs{\lambda_p}^t}-w_p(\sgn \lambda_p)^t} \le\beta \epsilon^t$ with  $\ds\epsilon \sett \frac{\max \{r_i,\cdots,r_q,\abs{\lambda_1},\cdots,\abs{\lambda_{p-1}} \}}{\abs{\lambda_p}} < 1$ and $\ds\beta\sett \sum_{i=1}^q\abs{d_i}+\sum_{j=1}^{p-1}\abs{w_j}$.
\end{proposition}

With this approximation and error in mind, we present upper bounds on $t_0$. The core idea is to bound $\beta \epsilon^t$ for large $t$, which is possible as $\epsilon < 1$. Based on this idea, we prove the following lemma for different signs of $w_p$ and $\lambda_p$. 

\begin{lemma}\label{lem:cutoff} Let $\epsilon,\beta$ be the same as in Proposition~\ref{prop:tightness}. For the first three cases below, there is a cutoff beyond which a $t_0$ can be found and an explicit formulation for $t_0$ exists. The final case yields a cutoff beyond which such a $t_0$ does not exist.
\begin{itemize}
\item $w_p\!>\!0,\lambda_p\!>\!0:$ $\forall t > \log_\epsilon\big(\frac{w_p}{\beta}\big)$, $g(t) > 0$. Thus, $t_0 = 1+\max \big\{\big\lceil \log_\epsilon\big(\frac{w_p}{\beta}\big) \big \rceil , 1\big\}$.
\item $w_p\!>\!0,\lambda_p\!<\!0:$ $\forall t > \frac12\log_\epsilon\big(\frac{w_p}{\beta}\big)$, $g(2t) > 0$. Thus, $t_0 = 2+2\max\big\{ \big\lceil \frac{1}{2}\log_\epsilon\big(\frac{w_p}{\beta}\big) \big \rceil , 1\big\}$.
\item $w_p\!<\!0,\lambda_p\!<\!0:$ $\forall t > \frac12\log_\epsilon\big(\frac{-w_p}{\beta}\big)$, $g(2t+1) > 0$. Thus, $t_0 = 2\max\big\{\big\lceil \frac{1}{2}\log_\epsilon\big(\frac{-w_p}{\beta}\big) \big \rceil , 1\big\} + 1 $.
\item $w_p\!<\!0,\lambda_p\!>\!0 :$ $\forall t>\log_\epsilon\big(\frac{-w_p}{\beta}\big)$, $g(t)<0$.
\end{itemize}
In the final case, instead of searching for $t_0$, we can directly search for $n_0$: search among $1\le n_0 \le \big\lfloor \log_\epsilon\big(\frac{-w_p}{\beta}\big)\big\rfloor$ for where $g(n_0)$ is positive and is maximized. If the search returns empty (meaning for all $t$, $g(t) < 0)$, then no $t_0$ or $n_0$ exists for this last case and $g$ is maximized at $\infty$ with value $0$. 
\end{lemma}
\begin{proof}{\informsproof}
We prove each bullet point respectively as follows.
\begin{itemize}
\item $w_p>0,\lambda_p>0$: 
$t_0 > \log_\epsilon\left(\frac{w_p}{\beta}\right) \implies \epsilon^{t_0} < \frac{w_p}{\beta}$. The last implication is because $\epsilon<1$. Since $w_p,\lambda_p>0$, Proposition~\ref{prop:tightness} gives $\abs{\frac{g(t_0)}{\abs{\lambda_p}^{t_0}}-w_p} < w_p$. This means $\frac{g(t_0)}{\abs{\lambda_p}^{t_0}} > 0$ whence $g(t_0) > 0$.
\item $w_p>0,\lambda_p<0$: 
We use a similar reasoning as above. If $t_0$ is even and $\frac{t_0}{2} > \frac12 \log_\epsilon\left(\frac{w_p}{\beta}\right)$, that is, $t_0 > \log_\epsilon\left(\frac{w_p}{\beta}\right)$ then Proposition~\ref{prop:tightness} gives us $\abs{\frac{g(t_0)}{\abs{\lambda_p}^{t_0}} -  w_p} \le \beta\epsilon^{t_0} < w_p$ which again implies that $g(t_0)>0$.
\item $w_p<0,\lambda_p<0$: 
Again, use similar reasoning as the first bullet point. $t_0$ being odd here ensures that $w_p(\sgn \lambda_p)^{t_0} = -w_p > 0$. Here we again have $t_0 > \log_\epsilon\left(\frac{-w_p}{\beta}\right)$. Proposition~\ref{prop:tightness} turns into $\abs{\frac{g(t_0)}{\abs{\lambda_p}^{t_0}}+w_p} < -w_p$ whence $g(t_0) > 0$.
\item $w_p<0,\lambda_p>0$: If $t>\log_\epsilon\left(\frac{-w_p}{\beta}\right)$ then $\epsilon^t<\frac{-w_p}{\beta}$ as earlier. For such $t\in\N$, Proposition~\ref{prop:tightness} gives $\frac{g(t)}{\abs{\lambda_p}^t} - w_p \le \beta\epsilon^t < -w_p$ whence $g(t) < 0$. So if $g(t) < 0$ for each $t\le t_0 \sett \left\lfloor \log_\epsilon\left(\frac{-w_p}{\beta}\right)\right\rfloor$ then $g$ is always negative on $\N$.
\end{itemize}
The claim of searching to find $n_0$ in the final case is straightforward from the description.
\informsendofproof \end{proof}




\textbf{Case 2}: $r_q>\abs{\lambda_p}$. Similar to case 1, we consider
$$\frac{g(t)}{r_q^t} = \sum_{i=1}^{q-1}d_i\Big(\frac{r_i}{r_q}\Big)^t\cos(t\theta_i+\eta_i)+ 
\sum_{j=1}^{p} w_j\Big(\frac{\lambda_j}{r_q}\Big)^t + d_q\cos(t\theta_q+\eta_q).$$
We again characterize the long-range behavior of the function with respect to $t$ using the triangle inequality, but the dominating term being $d_q\cos(t\theta_q+\eta_q)$ here.
\begin{proposition}\label{prop:tight2}
$\abs{\frac{g(t)}{r_q^t}-d_q\cos(t\theta_q+\eta_q)} \le \gamma C^t$ with $\ds C \sett \frac{\max\{r_1,\cdots,r_{q-1},\abs{\lambda_1},\cdots,\abs{\lambda_{p}} \}}{r_q} < 1$ and $\ds\gamma \sett \sum_{i=1}^{q-1} \abs{d_i} + \sum_{j=1}^p\abs{w_j}$. 
\end{proposition}

Unlike the previous case, the dominating term here is a cosine function with a certain phase $\eta_q$ and frequency $\theta_q$. In the previous case we could easily control the term dependent on $t$ (which was $\lp\sgn \lambda_p\rp^t$) by simply considering the two cases of the sign of $\lambda_p$. This is significantly harder in this case because of the behavior of cosine terms for input angles that are \emph{integral multiples} of the frequency $\theta_q$. To avoid irrationality issues (for example arithmetic modulo $\pi$), \textit{we measure angles in degrees}. We show that depending on the phase and frequency, one can carefully control the integers $t$ in order to get the expected behavior. More technically, we have the following bound for $t_0$ in this case.

\begin{lemma} \label{lem:cutoffcomplex}
Given angles are measured in degrees and $\theta_q = a/b \in\Q$ with $a,b\in\Z_{>0}$ and $\gcd(a,b)=1$, 
there exist integers $n\ne 0, l$ such that $an+360bl = \gcd(360,a) = g$, computable in $\cO(\log(\min(a,b)))$ time. Let $C,d_q,\gamma$ be as in Proposition~\ref{prop:tight2}. Then, for any integer $p \in (0, 90b)$,  $t_0$ is upper bounded by
$$\frac{n(p-bc)+\abs{n}\lp s_0+ \max\set{1,\lceil D \rceil }\rp360b }{ g},~~~\text{where}$$  
$s_0 = 1+\max\Big\{0,\Big\lceil\frac{\sgn(n)(cb-p)}{360}\Big\rceil\Big\}$, $D + s_0 = \frac{g\log_C\big(\frac{\abs{d_q}}{2\gamma}\big) - n(p-cb)}{360b\abs n}$  and $c = 135+ \lfloor \eta \rfloor -\sgn(d_q)\cdot 90$.
\end{lemma}

\begin{proof}{Proof Sketch}
We show in the full proof that $\exists k \in\Z_{>0}$ s.t. $\cos(k \theta_q+\eta_q) \ge 0.5$ and $k'\in\Z_{>0} $ $\cos(k' \theta_q+\eta_q) \le -0.5$. The full proof involves the use of the number theory results including Bezout's identity, parameters of which can be computed by Extended Euclid's algorithm, which takes $\cO(\log(\min(a,b)))$ in our case. Then, we use these results for the two different possible sign of $d_q$ in Proposition~\ref{prop:tight2} to get the bound above.
\informsendofproof \end{proof}


\textbf{Hardness}: The above results show that the steps required depend on the input problem, and can be arbitrarily large (e.g., for $\zeta$ nearly $1$ in Prop.~\ref{prop:n0}). We formalize this by a NP Hardness result. Before stating the result, we present a brief discussion about the decision version of an optimization problem and why both are equivalent when considering nondeterministic polytime (NP) hardness. Formally, an optimization problem has a polytime algorithm iff its corresponding decision version has a polytime algorithm~\citep{cormen2022introduction}. In brief, for an optimization problem of $\max F(x)$ with the constraint that $x\in X$, 
the corresponding decision version is to decide whether there exists $x\in X$ satisfying $F(x) \ge \alpha$ for given input $\alpha\in \R$.
In fact, if there is a polytime (in size of inputs) algorithm $\cA$ to answer the latter, for any inputs $\alpha, F,X$, then given any additive accuracy parameter $\epsilon>0$, one can use bisection (binary search) by calling $\cA$ about $\cO\lp\log \frac1\epsilon\rp$ times (overall still polytime) in order to solve the optimization problem upto additive $\epsilon-$accuracy. 
The other way is not hard to see and can be found in textbooks~\citep{cormen2022introduction}.

In this work, the decision version of \ref{eq:re} asks whether there exists $t\in \N$ such that $\inner{\bm c}{M^t\bm x} \ge \alpha$, given inputs $n\in\N,\bm c\in \Q^n, \bm x\in \Q^n, M\in \Q^{n\times n}, \alpha\in \Q$ with $\rho(M)<1$. We use $\Q$ instead of $\R$ to stick to the Turing model of computation. Indeed, the next NP Hardness result shows that the~\ref{eq:re} problem for infinite support is not easy in general, unless $\text{P} = \text{NP}$.

\begin{theorem} \label{thm:nphard}
Given $n\in\N$, $M\in\Q^{n\times n}$ with $\rho(M)<1$, $\bm x\in \Q^n$, $\bm c\in \Q^n$, $\alpha\in\Q$, it is NP-hard to decide whether there exists $t\in \N$ such that $\inner{\bm c}{M^t\bm x} \ge \alpha$.
\end{theorem}

\begin{proof}{\informsproof}
The following decision problem is known to be NP-hard from~\cite{dir}: ({\sf{DIRCYC}}) Given a directed graph $G$ on $n$ nodes, is there an integer $t^{*} \in \N$ such that $G$ has no directed path of length $t^{*}$ from node $1$ to node $n$? 

We will show a reduction from this problem to the problem of our interest.
Let $n\in \N$ and graph $G$ be an input to {\sf{DIRCYC}}. We describe an input to our problem as follows. Let $A$ be the adjacency matrix of $G$. It is well known that $p_i^{(j)} \sett \bm e_1^\top A^j\bm e_i=(A^j)_{1i}$ is the number of directed walks of length $j$ from $1$ to $i$ in $G$. Let $r = 1 + \max\limits_j\sum\limits_iA_{ij}$. Proposition~\ref{prop:normbound} (in appendix) for $p=1$ implies $\rho(A) \le r-1$. So $\frac1rA$ has all eigenvalues of size $<1$.
Take inputs as $n, M = \frac1rA, \bm x = \bm e_n, \alpha = 0, \bm c=-\bm e_1$.  
Note that for any $t\in\N, \inner{\bm c}{A^t\bm x} = p_n^{(t)}$ for the aforementioned inputs.

We claim $\exists ~t^{*}\in \N$ for which $G$ has no directed path from $1$ to $n$ of length $t^{*}$ iff $\exists~t^*\in\N$ such that $\inner{\bm c}{M^{t^*}\bm a} \le \alpha$. Indeed, $\exists ~t^*\in \N$ such that $\inner{\bm c}{M^{t^*}\bm a} \ge \alpha \iff -\frac{p_n^{(t^*)}}{r^{t^*}} \ge 0 \iff p_n^{(t^*)} \le 0 \iff p_n^{(t^*)} = 0 \iff G$ has no directed path from $1$ to $n$ of length $t^{*}$.
\informsendofproof \end{proof}

Given the above, unsurprisingly, the lemma below shows that there exists instances for which the~\ref{eq:re} problem for infinite support is not solvable in polytime following our approach. For special choices of inputs $\bm c, \bm x, M$, the objective $\inner{\bm c}{M^t\bm x}$ can be controlled so that the \emph{first time} $t^*$ that the objective becomes positive is arbitrarily large. However, we also show in \iftoggle{isORtext}{the supplementary material}{Appendix~\ref{sec:specialcase}} that a constant upper bound for $t^*$ is possible for the special case of $n=2$.
\begin{lemma}\label{lem:blow}
For every $k=1,2,\cdots$, there are inputs $\bm c_k,\bm x_k\in \R^n, M_k \in \R^{n\times n}$ with $\rho(M_k) < 1$ such that $\inf\{t\!\in\!\N~|~\inner{\bm c_k}{M_k^t\bm x_k} \!>\! 0 \} \ge k$. In particular, the infimum exists (for these special inputs).
\end{lemma}

\subsection{Infinite Support: Distributional Robust Cost Estimation}
The NP hardness result for \ref{eq:re} already shows that the~\ref{eq:dre} problem is hard, as~\ref{eq:re} is a special case of~\ref{eq:dre}. In fact, note that representing the distribution $\hat{\bm p}$ over infinite support is possible on a computer only as a parametric distribution. Here, noting that a geometric distribution with parameter $\rho$ captures the first occurrence of an event, we show results for distributions represented by a geometric distribution with the success event indicating the stoppage of the GAS system. We rely on the fact that the Wasserstein-1 distance between two geometric distributions can be written in terms of the parameters~\cite{de20211}. Similar results as below can be obtained for other parametric distributions, such as Poisson or negative binomial, that admit a closed form formula for Wasserstein-1 distance in terms of the parameters.

Let $\hat{\rho}$ be the parameter for geometric distribution $\hat{\bm p}$, and we restrict the distributions in the ambiguity set of Equation~\ref{eq:ambiguity} to be from the geometric family, thus, we write $\bm q_\rho$ to indicate the parameter $\rho$ for any $\bm q \in \mathcal{P}$. By geometric distribution $q_\rho(t) = (1 - \rho)^{t-1}\rho$. Then, from standard results~\cite{de20211}, we get $
W_1(\bm q_\rho, \hat{\bm p}) = \vert 1/\rho - 1/\hat{\rho} \vert  
$, which further implies the restriction that $\frac{\hat{\rho}}{1 + \hat{\rho \xi}} \leq \rho \leq \frac{\hat{\rho}}{1 - \hat{\rho \xi}}$. Then, using the fact that $g(t)$ converges to $\bm 0$, we can choose a $n_0$ and optimize for $
\sum_{t = 1}^{n_0} g(t) q_\rho(t)
$ to obtain the following arbitrary additive approximation.
\begin{proposition} \label{prop:parametric}
Let $\ds n_0 =  \big\lceil\log_{\zeta}{\frac{\epsilon}{\sum\abs{d_i}+\sum \abs {w_j}}}\big\rceil$, where $\zeta = \max\set{\abs{\lambda_p},r_q}$, for small $\epsilon$. Let $\delta = (1 - \rho)^{n_0}$, then $\ds \abs{ 
\sum_{t = 1}^{\infty} g(t) q_\rho(t)
-
\sum_{t = 1}^{n_0} g(t) q_\rho(t) } \leq \epsilon \delta
$.
\end{proposition}

\begin{proof}{\informsproof}
We obtain $n_0$ directly from application of from Proposition~\ref{prop:n0} by replacing $g(t_0)$ by $\epsilon$. Then, the CDF for $\ds \bm q_\rho$ is $1 - (1 - \rho)^t$. Thus, $\ds \abs{\sum_{t = n_0 + 1}^{\infty} g(t) q_\rho(t)} \leq \epsilon \sum_{t = n_0 + 1}^{\infty}  q_\rho(t) = \epsilon \delta
$. 
\informsendofproof \end{proof}

We note that $\ds \sum_{t = 1}^{n_0} g(t) q_\rho(t) $ may be a non-convex function in $\rho$ and we can use projected gradient ascent with random restarts to aim to find the maximum. 

\section{Experiments}\label{sec:experiments}
While our work is mainly theoretical in nature, we provide some experiments supporting the theory. In the main paper, we focus on two experiments: (1) synthetic experiments showing how our approach scales with state space, (2) CSOC overtime alert processing with partial data from a real CSOC, and (3) disease spread modeling using numbers taken from~\cite{10.1093/oso/9780192896087.003.0022}. Additional experiments showing the computational advantage of converting a Markov Chain to a GAS system, as presented in Section~\ref{sec:relation}, are in 
\iftoggle{isORtext}{the supplementary material.}{Appendix~\ref{sec:appendixexperiments}}. 
All experiments were run on Apple M2 Chip with 8 CPUs over 2.4 GHz each and 8GB RAM.
\iftoggle{isOR}{}{\begin{table}[t]
\centering
\noindent
\begin{minipage}[t]{0.49\textwidth}
  \centering
  \captionof{table}{SIR Transitions} \label{tab:sir}
  \begin{tabular}{cccc}
    \toprule
    From & to S & to I & to R \\
    \midrule
    S & 0.2 & 0.8 & 0\\
    I & 0 & 0.5 & 0.5 \\
    R & 0.1 & 0 & 0.9 \\
    \bottomrule
  \end{tabular}
\end{minipage}%
\hfill
\begin{minipage}[t]{0.49\textwidth}
  \centering
  \captionof{table}{SVIR Transitions} \label{tab:svir}
  \begin{tabular}{ccccc}
    \toprule
    From & to S & to V & to I & to R \\
    \midrule
    S & 0.1 & 0.1 & 0.8 & 0 \\
    V & 0.1 & 0.9 & 0 & 0 \\
    I & 0 & 0 & 0.5 & 0.5 \\
    R & 0.1 & 0 & 0 & 0.9 \\
    \bottomrule
  \end{tabular}
\end{minipage}
\hfill

\end{table}}

\iftoggle{isOR}{}{\begin{figure}
\begin{center}
 \includegraphics[scale=0.75]{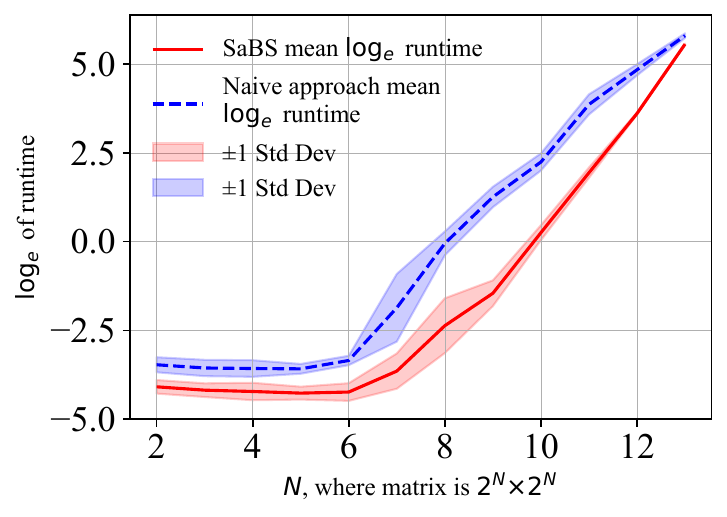}   
\end{center}
\caption{Runtime over 30 random instance with varying size.} \label{fig:runtime}
\end{figure}}
Our synthetic experiments are for the~\ref{eq:re} problem with an uncertain time horizon $\mathcal{T} = [T] = [10000]$. The results are based on runs of the SaBS algorithm and a baseline naive approach (Section~\ref{sec:finitesupportrce}) on 30 randomly generated input matrices $M$ of a given size $2^N \times 2^N$. We increase the size of $N$, with a cutoff of the model size that runs within an hour. Figure~\ref{fig:runtime} shows the runtime (in log scale of seconds) against $N$.  The results align with our claim that our approach provides scalability for $T$ larger than the matrix size $2^N$. The results also show that SaBS can handle a large state space and a large $T$ on a regular laptop much better than the naive approach. 

Both our semi-synthetic data based experiments compare the performance when estimating costs with and without considering distributional uncertainty, given $k$ samples $t_1, \ldots ,t_k$ of the number of timesteps from past observations or domain experts. Without distributional uncertainty, we estimate $\hat{t}$ as the average of the time samples and just compute \emph{empirical expected cost} $\hat{C}$ as $\hat{C} = M^{\hat{t}} x_0$. With distributional uncertainty, we form an empirical distribution $\hat{\bm p}$ from the $k$ time samples and solve for~\ref{eq:dre} with a Wasserstein ball of size $\xi$ around the empirical distribution to get the~\ref{eq:dre} solution as \emph{expected cost} $C$. Also, we run the chain (sample initial state and then sample next state every timestep) $k$ times, once for each of the $k$ time sample and stop at the timestep $t_k$ for the $k$-th sample. At the end of each run, we record the cost as given by the final state. Thus, we obtain $k$ possible values of costs. We then estimate how likely is it that the true costs will exceed $\hat{C}$ and 
$C$, by calculating the percentage of the $k$ costs that lie above these costs, indicated as $\% > \hat{C}$ and $\% > C$ respectively, in our results tables.

\iftoggle{isOR}{}{\begin{table}[t]
\centering
\begin{minipage}[t]{0.44\textwidth}
  \centering
  \captionof{table}{Solution Characteristics for Health Problem} \label{tab:sol}
  \begin{tabular}{ccc}
    \toprule
    Characteristic & SIR & SVIR \\
    \midrule
    Emp. Cost $\hat{C}$ & 0.75 & 0.69 \\
    \%$ > \hat{C}$ & 56 & 54   \\
    \sf{DRCE} Cost $C$ & 1.96 & 1.03  \\
    \%$ > C$ & 23 & 15  \\
    \bottomrule
  \end{tabular}
\end{minipage}
\begin{minipage}[t]{0.52\textwidth}
\centering
  \captionof{table}{Solution Characteristics for Overtime CSOC problem} \label{tab:csoc}
  \begin{tabular}{ccc}
    \toprule
    Characteristic & $\xi = 16$ & $\xi = 32$ \\
    \midrule
    Emp. Cost $\hat{C}$ & 0.73 & 0.73 \\
    \%$ > \hat{C}$ & 41  & 41 \\
    \sf{DRCE} Cost $C$  & 1.33  & 1.9 \\
    \%$ > C$ & 12 & 2  \\
    \bottomrule
  \end{tabular}
\end{minipage}
\end{table}}

Our first experiment utilizes data collected from a real-world operational CSOC (anonymized due to double blind). The alert-handling process begins with the primary inspection of alerts generated by intrusion detection systems, during which analysts determine whether each alert is benign or malicious, an assessment that directly influences the organization’s overall security posture.
On average, two analysts perform primary inspection for cyber alerts coming in at the average Poisson rate of $35$ alerts per hour per person, with an average service rate of $34$ alerts per hour per person. The service rate follows a uniform distribution over range $34 \pm 3.4$ that allows 10\% deviation around the mean of $34$. Considering a timestep as $30$ seconds and possible alert queue sizes $\{0,1,\ldots, U=100\}$, we can form a transition matrix $M'$ for the discrete time Markov chain defined by each analyst's queue, where the queue for each analyst is separate. For the chain $M'$ for an analyst, the shift begins in a state $(0)$, indicating that there are $0$ alerts in the queue at the start of the shift. Then, there is an end of shift (960 timesteps) distribution over the number of alerts pending for an analyst, which we denote as $ \bm x_0$ as it serves as the initial state for our main object of study, the \emph{overtime Markov chain}. This overtime Markov chain, given by $M$, starts after the regular shift and is same as one in the regular time chain but with an arrival rate of $0$, since new alerts go to next shift employees. The chain $M$ is not ergodic due to state $(0)$ being a sink state, but it is straightforward to verify that it has one eigenvalue equal to $1$ and rest with magnitude less than $1$, resulting in the unique stationary distribution that puts all probability mass on $(0)$ at infinity. In our consideration, the end of overtime, or the time horizon of the overtime Markov chain, happens at some finite number of timesteps. The cost at the end of overtime is dependent on the number of alerts left in the queue, namely 0 cost for 0 alerts, 0.5 cost for 1 alert, and then linearly scaled to the highest cost of 1 for the maximum number of $U$ alerts. The total cost is the sum of the costs for each analyst. We assume that the number of timesteps in overtime (or the duration of overtime per person) could vary from 1 (30 sec) to 120 (1 hour) with a mean of 61 timesteps. With $k=100$ time samples we get $\hat{t} = 61$. The results are shown in Table~\ref{tab:csoc}, showing that worst case costs $C$ can be much higher than $\hat{C}$, which is computed assuming $\hat{t}$ overtime timesteps. As clearing the remaining queue would cost in terms of additional resource and such budgeting decisions need to be taken a priori, it is important in security settings to take into account worst case scenarios before committing to extra resources in reserve.

Our second experiment is based on transition matrix numbers from~\cite{10.1093/oso/9780192896087.003.0022} modeling a SIR disease spread within five people. We add a vaccinated state to model SVIR disease spread and compare the vaccinated vs non-vaccinated scenario. The transitions are shown in Table~\ref{tab:sir} and ~\ref{tab:svir}. The full Markov chain states describe the joint state of all individuals (e.g., SSSSS means all five are Susceptible). We associate a cost of 1 with Infected, and 0 otherwise, e.g., state SIISR has a cost of 2. We assume that everyone is Susceptible initially for the SIR model and considering practicalities such as vaccine hesitancy, 40\% of the population is Susceptible and 60\% Vaccinated initially for the SVIR model. We are interested in understanding the cost after 8 days, which nominally suggests that we should take 8 timesteps in the Markov chain. However, based on our motivation that a timestep a random variable over time instead of fixed 24 hours, we assume that the number of timesteps could vary from 1 to 15 with mean as 8. With $k=100$ time samples we get $\hat{t} = 8$.
The numbers for this experiment are in Table~\ref{tab:sol}. As expected,~\ref{eq:dre} solution cost is more conservative and the chance of actual cost exceeding $C$ is smaller than that for 
$\hat{C}$; this allows for more robust decision making. Also, the~\ref{eq:dre} 
 solution shows a bigger difference in cost $C$ between the SIR and SVIR cases, which reveals that~\ref{eq:dre}
 can bring out the difference between vaccinating or not more starkly and with higher confidence.



\section{Conclusion}\label{sec:conclusion}
We proposed a pertinent distributional robust cost estimation problem in GAS systems with an uncertain time horizon, and presented theoretical results to solve it that also yielded fundamental general theory results. The robustness is particularly relevant in resource allocation or policy decision making in critical settings such as cybersecurity operations or health care security. A number of promising research directions can be pursued, such as exploring different ambiguity sets, time horizon robustness in control problems, and combining with uncertainty in other problem parameters, enabling resilient, data-driven strategies that enhance security, operational reliability, and the protection of critical infrastructure and population-level resources.



\bibliographystyle{plain}
\bibliography{references}
\appendix



\section{Missing and Full Proofs} \label{sec:proofs}


\subsection{Proof of Corollary~\ref{cor:structure}}
\begin{proof}{\informsproof}
Observe that $\bm x_i - \bm \pi \in H_{n-1}$ (i.e., the vector components sum to 0). Thus, by second property of the theorem, $\overline M \bm v_i = \overline M A (\bm x_i - \bm \pi) = AM(\bm x_i - \bm \pi) = A(\bm x_{i+1} - \bm \pi) = \bm v_{i+1}$. Next, $B \bm v_i = BA(\bm x_i - \bm \pi) = \bm x_i - \bm \pi$ because $BA$ is identity on $H_n$.
\informsendofproof \end{proof}

\subsection{Proof of Proposition~\ref{prop:GAS}}
\begin{proof}{\informsproof}
First, we show that $\overline M$ is Lyapunov stable at $\bm 0$. Note if $\bm x\in\R^{n-1}$ then $\norm{\overline M^t\bm x}1 = \norm{AM^tB\bm x}1 \le \norm{A}{1}\norm{M}{1}^t\norm{B}1\norm{\bm x}{1} = \norm{A}{1}\norm{B}{1} \norm{\bm x}{1}$. For any given $\epsilon>0$ take $\delta\sett \dfrac{\epsilon}{\norm{A}{1}\norm{B}{1}}$. So if $\norm{\bm x}{1} < \delta$ then $\norm{\overline M^t\bm x}{1} < \norm A1\norm B1\delta = \epsilon$.
Then, $\norm{AM^tB\bm y}{1} \le \norm{A}{1}\norm{M^t\bm x-\bm \pi}{1}\stackrel{t\to\infty}{\longrightarrow} \bm 0$. Therefore $\overline M$ has all eigenvalues with magnitude $<1$. 
\informsendofproof \end{proof}

\subsection{Proof of Proposition~\ref{prop:initial}}
\begin{proof}{\informsproof}
The following directly follow from the definitions of $\bm x_t$ and $\mathcal{U}$.
\begin{align*}
\max_{\bm x_0 \in \mathcal{U}, \bm q \in \mathcal{P}} \;  \sum_{t \in \mathcal{T}}  q_t \inner { \bm c} {\bm x_t } 
& = \max_{\bm x_0 \in \mathcal{U}, \bm q \in \mathcal{P}} \; \sum_{t \in \mathcal{T}}  q_t  \inner { \bm c} {M^t {\bm x}_0} \\
& = \max_{\alpha_i, \bm q \in \mathcal{P}} \; \sum_{t \in \mathcal{T}}  q_t  \Big ( \inner { \bm c} {M^t \hat{\bm x}_0} + \sum_{i=1}^U \alpha_i \inner { \bm c} {M^t {\bm u}_i} \Big) 
\end{align*}
Next, let $\bm q^*$ and $\alpha^*_i$ be a solution of the optimization in the last line above. Then, observe that
\begin{align*}
\sum_{t \in \mathcal{T}}  & q^*_t  \Big ( \inner { \bm c} {M^t \hat{\bm x}_0} + \sum_{i=1}^U \alpha^*_i \inner { \bm c} {M^t {\bm u}_i} \Big)
 = \sum_{t \in \mathcal{T}}  q^*_t \inner { \bm c} {M^t \hat{\bm x}_0} + \sum_{i=1}^U \alpha^*_i \sum_{t \in \mathcal{T}}  q^*_t  \inner { \bm c} {M^t {\bm u}_i}  
\end{align*}
Let $ J = \arg\!\max_{i} \sum_{t \in \mathcal{T}}  q^*_t  \inner { \bm c} {M^t {\bm u}_i} $. It is straigtforward to check that $\alpha^*_i = 0$ for $i \notin J$. In fact, any $\alpha$ values with $\alpha_i = 1$ for any $i \in J$ is a feasible solution that yields the same optimal value. Thus, there exists an optimal solution that has $\alpha_i = 1$ and vertex $u_i$ selected. This leads to the simple algortihm enumerated below 
\begin{enumerate}
\item Iterate over all vertices $u_i$ and solve $\max_{\bm q \in \mathcal{P}} \; \sum_{t \in \mathcal{T}}  q_t  \inner { \bm c} {M^t (\hat{\bm x}_0 + u_i)}$.
\item Choose the best solution value from the above loop.
\end{enumerate}
This algorithm is exactly same as $\max_{i \in [U]} \max_{\bm q \in \mathcal{P}}  \sum_{t \in \mathcal{T}} q_t \langle \bm c, M^t (\hat{\bm x}_0 + \bm u_i )\rangle$.
\informsendofproof \end{proof}


\subsection{Proof of Proposition~\ref{prop:wasserstein-norm}}
\begin{proof}{\informsproof} Let's write $\norm{\cdot}{}$ for $\norm{\cdot}{W,d}$. Call $C \sett \set{\bm x\in\R^T\st \abs{x_i-x_j}\le d_{ij}\fa i,j}$. Note that $\bm x\in C\iff -\bm x\in C$.\\
(Homogeneity) If $\bm \mu\in H_T,\lambda\in\R$ then 
$\ds\norm{\lambda\bm\mu}{} = \max_{\bm x\in C}\lambda\bm \mu^\top \bm x = \abs\lambda \max_{\bm x\in C}\bm \mu^\top \bm x = \abs\lambda \norm{\bm \mu}{}$ because $\bm x$ can be replaced with $-\bm x$ depending on the sign of $\lambda$.\\
(Positive definiteness) Since $\bm x\in C\iff -\bm x\in C$, we see that $\norm{\bm \mu}{}\ge 0$ always. Suppose $\norm{\bm \mu}{} = 0$ for some $\bm\mu\in G_T$. Then $\bm \mu \le 0\fa \bm x\in C$ by definition. Let $\alpha_i \sett \min_{j\in[T]\smallsetminus\set{i}} d_{ij}$. Clearly each $\alpha_i>0$ and each $\bm \omega_i \sett \alpha_i\bm e_i\in C$. The latter is true because for $k\ne j, j\ne i$ $\abs{(\bm \omega_i)_k-(\bm\omega_i)_j} = \begin{cases}\alpha_i&\text{if } k=i\\
0&\text{otherwise}
\end{cases} \le d_{kj}$. It follows that $0 \ge \bm \mu^\top \bm \omega_i = \alpha_i\mu_i$. But $\alpha_i>0$ whence $\mu_i\le 0\fa i$. The only $\bm \mu\in H_T$ with all non-positive coordinates is $\bm \mu=\bm 0$.\\
(Triangle inequality) Take any $\bm \mu,\bm \nu\in H_T$. Then for each $\bm x\in C$ we have $\ds (\bm \mu + \bm \nu)^\top \bm x = \bm \mu^\top\bm x + \bm\nu^\top\bm x \le \max_{\bm y\in C}\bm\mu^\top\bm y + \max_{\bm z\in C}\bm\nu^\top\bm z = \norm{\bm \mu}{} + \norm{\bm \nu}{}$. It thus stands that $(\bm \mu + \bm \nu)^\top \bm x \le \norm{\bm \mu}{} + \norm{\bm\nu}{}\fa \bm x\in C$. Taking $\max$ over $\bm x\in C$ gives $\norm{\bm\mu + \bm \nu}{} \le \norm{\bm \mu}{} + \norm{\bm \nu}{}$.
\informsendofproof \end{proof}

\subsection{Proof of Proposition~\ref{prop:tightness}}
\begin{proof}{\informsproof}
From definition of $g(t)$ and triangle inequality, we get 
\begin{align*}
\ds\abs{\frac{g(t)}{\abs{\lambda_p}^t}-w_p\left(\sgn \lambda_p\right)^t} & = \abs{\sum_{i=1}^q \frac{d_ir_i^t}{\abs{\lambda_p}^t} \cos(t\theta_i+\eta_i) + \sum_{j=1}^{p-1}\frac{w_j\lambda_j^t}{\abs{\lambda_p}^t} } \\
& \le \sum_{i=1}^q \abs{d_i}\abs{\frac{r_i^t}{\lambda_p^t}} + \sum_{j=1}^{p-1}\abs{w_j}\abs{\frac{\lambda_j^t}{\lambda_p^t}} 
 \le \left(\sum_{i=1}^q \abs{d_i} + \sum_{j=1}^{p-1} \abs{w_j}\right)\epsilon^t = \beta\epsilon^t
\end{align*}
\informsendofproof \end{proof}


\subsection{Proof of Proposition~\ref{prop:tight2}}
\begin{proof}{\informsproof}
From definition of $g(t)$ and triangle inequality, we get 
\begin{align*}\ds\abs{\frac{g(t)}{r_q^t}-d_q\cos(t\theta_q+\eta_q)} & = \abs{\sum_{i=1}^{q-1} \frac{d_ir_i^t}{r_q^t} \cos(t\theta_i+\eta_i) + \sum_{j=1}^{p}\frac{w_j\lambda_j^t}{r_q^t} } \\
& \le \sum_{i=1}^{q-1} \abs{d_i}\abs{\frac{r_i^t}{r_q^t}} + \sum_{j=1}^{p}\abs{w_j}\abs{\frac{\lambda_j^t}{r_q^t}}  \le \left(\sum_{i=1}^{q-1} \abs{d_i} + \sum_{j=1}^{p} \abs{w_j}\right)C^t = \gamma C^t
\end{align*}
\informsendofproof \end{proof}

\subsection{Proof of Lemma~\ref{lem:cutoffcomplex}}
\begin{proof}{\informsproof}
We first want to show that $\exists k \in\Z_{>0}$ s.t. $\cos(k \theta_q+\eta_q) \ge 0.5$ and $k'\in\Z_{>0} $ $\cos(k' \theta_q+\eta_q) \le -0.5$, which can be found in $\cO(\log(\min(a,b)))$.

For ease of notation, write $\theta,\eta,d$ for $\theta_q,\eta_q,d_q$ respectively. Let's first do the calculation \emph{assuming $d>0$}. Let $\epsilon$ be such that $\cos^{-1} \epsilon = 50$. If we find some integers $k>0,m$ such that $k \theta + \eta \in (360m -50, 360m + 50)$, this implies that $\cos (k \theta + \eta ) > \epsilon > 0.5$. Next,
$$
k \theta + \eta \in (360m -50, 360m + 50) \implies k \theta  \in (360m -50 - \eta, 360m + 50 - \eta)
$$
Then, we need to find $k>0,m$ such that
$
 k \theta  \in (360m -50 - \eta, 360m + 50 - \eta) 
$.

Consider $i = \lfloor  \eta \rfloor$. It is enough to find $k$ such that (note inclusive below)
$$
k \theta  \in [360m -45 - i, 360m + 45 - i]
$$
The above is equivalent to $k\theta+45+i-360m\in[0,90]$.

Let us use a shortform $c = 45 + i$.
Let $\theta=a/b$ and $0 \leq z \leq 90$. Choose $z = p/b$, for any $0< p < 90b$. 
Then, we want to show that $k \theta + c = 360m + z$ for some positive integer $k$ and an integer $m$. 

From now on we treat $k,m$ as integer variables which we solve for. Substituting the fractions, we need to solve $ k\cdot a - m\cdot 360b = p-cb$. Since $\gcd(a,b)=1$ so $\gcd(a,360b)=\gcd(a,360)=g$. Then, $a/g$ and $360b/g$ are relatively prime, and $g\divides p-cb$ (read $\divides $ as divides).
By Bezout's identity, there are integers $n,l$ such that $n\cdot a/g + l\cdot360b/g=1$. We claim that $n\ne 0$ because if it were, then $l=360b/g=1$ so that $g\divides 360$ (as $\gcd(g,b) \le \gcd(a,b)=1$), which would mean that $g=360$ further implying $a\ge 360$ and $b=1$ making $\theta\ge 360$ but we had assumed $\theta<360$.

Say $n>0:$ Choose integer $s\ge 0$ such that $\alpha_s\sett(p-cb)/g+s\cdot360b/g>0$. Then $n\alpha_s\cdot a/g + l\alpha_s\cdot360b/g=\alpha_s$. \par
Say $n<0:$ Choose integer $s\le 0$ such that $\alpha_s \sett (p-cb)/g+s\cdot360b/g<0$. Then $n\alpha_s\cdot a/g + l\alpha_s \cdot360b/g = \alpha_s$.

Combined, we can say that with $\ds s_0 \sett 1+\max\set{0,\left\lceil\frac{\sgn(n)(cb-p)}{360}\right\rceil}$ and $f\ge 0$ we have that $\alpha=\alpha_{\sgn(n)(s_0+f)}$ satisfies $n\alpha\cdot a/g + l\alpha\cdot360b/g=\alpha$ with $n\alpha>0$. Thus we have found infinitely many solutions for $(k,m)$, namely $\set{\left(n\alpha_{\sgn(n)(s_0+f)}, s_0+\sgn(n)t - l\alpha_{\sgn(n)(s_0+f)}\right)\st f\in\Z, f\ge 0}$. The important point is that each aforementioned solution for $k$, that is $n\alpha_{\sgn(n)(s_0+f)}$, is positive and is a strictly increasing sequence. In fact, the expression equals $\ds n\alpha_{\sgn(n)(s_0+f)} = n(p-cb)/g + \abs{n}(s_0+f)360b/g$ for $f\ge 0$ and is always positive. Define $k_f\sett n(p-cb)/g + \abs{n}(s_0+f)360b/g$. They satisfy that $\cos(k_f\theta+\eta) \ge 0.5$ whence by Proposition~\ref{prop:tight2}, we conclude that for any positive integer $f \ge \frac{g\log_C\left(\frac{d}{2\gamma}\right) - n(p-cb)}{360b\abs n}-s_0$, we have $g(k_f) > 0$ because such a choice of $f$ makes $\gamma C^{k_f} < \frac{d}{2}$.

\emph{If $d<0$}, we seek to solve the same problem but with $\cos(k\theta+\eta)<-0.5$ and this is equivalent to $\cos(k\theta+180+\eta) > 0.5$. That is, we can repeat the above steps by replacing $\eta$ with $\eta+180$. So we want to find $t$ for which $\gamma C^t < \frac{-d}{2}$ and this is achieved with $k_f = n(p-c'b)/g+\abs{n}(s_0'+f)360b/g$ where $c'=45+i+180$ and $s_0'$ is the same expression as $s_0$ above, but $c$ replaced with $c'$.

Putting these together, our $t_0$ takes the form
$$\frac{n(p-bc)+\abs{n}\left(s_0+ \max\set{1,\left\lceil \frac{g\log_C\left(\frac{\abs d}{2\gamma}\right) - n(p-cb)}{360b\abs n}-s_0\right\rceil}\right)360b}{g}$$ where $s_0 = 1+\max\set{0,\left\lceil\frac{\sgn(n)(cb-p)}{360}\right\rceil}$ and $c = 135+i-\sgn(d)\cdot 90$.

In the above solving for the $n,l$ using Bezout's identity can be done by Extended Euclid's algorithm, which takes time of the order of the smaller of bit representation of $a/g$ and $360b/g$. With a worst case of $g=1$, this is $\cO (\log (\min (a, b) )$, which is the dominating computation. \informsendofproof \end{proof}


\subsection{Proof of Lemma~\ref{lem:blow}}
\begin{proof}{\informsproof}
We start with an $n=2$ proof and will extend it to general $n$ later. Consider $\alpha_k,\theta_k \in \Q$ and inputs $\bm c_k,\bm x_k,M_k$ as follows: $\bm c_k=(1,0), \bm x_k=(\cos\alpha_k,\sin\alpha_k), M_k=\frac12\begin{bmatrix}\cos\theta_k&-\sin\theta_k\\\sin\theta_k&\cos\theta_k\end{bmatrix}$. The factor $\frac12$ was used because $2M_k$ has all distinct eigenvalues of size $1$, whence $M_k$ has all eigenvalues of size $<1$. In other words, $\rho(M_k) < 1$. Thus, $2^{t}g(t) = 2^{t}\bm c_k^\top M_k^t\bm x_k =  \cos\alpha_k \cos t\theta_k- \sin\alpha_k\sin t\theta_k = \cos (\alpha_k + t\theta_k)$.

We show there exist $\alpha_k,\theta_k \in \Q$ such that $g(t)$ is negative for any $t\in [k]$.
It is well known from arctan series (Taylor expansion of $\tan^{-1}$) that $\ds \frac\pi4=1-\frac13+\frac15-\frac17+\cdots$. Take $x_i \sett \dfrac{1}{2i-1}$. Then, it is also known that $$\ds\abs{\frac{\pi}{4} - \sum_{i=1}^k(-1)^{i+1}x_i} \le x_{k+1} = \frac{1}{2k+1}.$$ 
So $\ds\abs{\frac{\pi}{2} - \sum_{i=1}^k(-1)^{i+1}2x_i} \le 2x_{k+1} = \frac{2}{2k+1}$. 
Consider the terms $$\ds s_i \sett 2x_{2i-1} - 2x_{2i} = \frac{2}{4i-3} - \frac{2}{4i-1} = \frac{4}{(4i-1)(4i-3)} > 0.$$ 
This gives 
$$\ds\abs{\frac{\pi}{2} - \sum_{i=1}^{k}s_i} = \abs{\frac{\pi}{2} - \sum_{i=1}^{2k}(-1)^{i+1}2x_i} \le 2x_{2k+1} = \frac{2}{4k+1}.$$ 
This proves that $\lim_{k\to\infty}\sum_{i=1}^k s_i = \pi/2$ and the convergence is monotone. So 
\begin{equation} \label{eq:upperb}
\frac{\pi}{2} - \sum_{i=1}^{k}s_i = \abs{\frac{\pi}{2} - \sum_{i=1}^{k}s_i} \le \frac{2}{4k+1}.\end{equation}

Let $\alpha_k\sett \sum_{i=1}^k s_i$ and $\ds\theta_k\sett 2/(4k+1)$. Since $\alpha_k\uparrow \pi/2$, we have $\cos\alpha_k>0, \sin\alpha_k>0$
\begin{itemize}
\item From Equation~\ref{eq:upperb}, $\pi/2 < \alpha_k +\theta_k$.
\item If $t>1$ is a positive integer, then $t\theta_k +\alpha_k > \theta_k+\alpha_k$ because $\theta_k>0$.
\item Finally, $k \theta_k = \ds\frac{2k}{4k+1} < \frac{1}{2}<\frac{\pi}{2}<\left(\frac{\pi}{2}-\alpha_k\right)+\frac\pi2 =\pi - \alpha_k \implies k\theta_k + \alpha_k < \pi.$
\end{itemize}
The above bullet points prove that $$\pi/2 < \alpha_k+\theta_k < \alpha_k+2\theta_k < \cdots < \alpha_k+k\theta_k < \pi.$$ We know that $\cos$ is strictly decreasing over $[0,\pi]$ and thus $$0 > \cos(\alpha_k+t\theta_k) = 2^{t}g(t) > -1 \mbox{ for all integers } t\in[k].$$ We have thus shown that the first positive integer $t$, if any, satisfying $\cos(\alpha_k+t\theta_k)>0$ must satisfy $t>k.$  In fact, $t=9k$ is such a candidate. Indeed, $$\frac{5\pi}{2} > \pi + 4 > \pi + \frac{16k}{4k+1} = \pi + 8k\theta_k > \alpha_k + 9k\theta_k$$ and 
$$\alpha+9k\theta_k > \frac\pi2-\theta_k+9k\theta_k = \frac{18k-2}{4k+1} + \frac{\pi}{2} \stackrel{[\because k\ge 1]}{\ge} \frac{16}{5} +\frac\pi2 > \frac{3\pi}2.$$
Since $\alpha_k+9k\theta_k$ is between $\frac{3\pi}{2},\frac{5\pi}{2}$, we must have $\cos(\alpha_k+9k\theta_k) > 0$. This proves our desired result for $n=2$.

Take $\tilde M_k = \text{diag}(M_k, 2^{-1}, 3^{-1}, \ldots, (n-1)^{-1}) \in\R^{n\times n}, \tilde{\bm x}_k = (\bm x,0,\cdots,0)\in\R^n, \tilde{\bm c}_k=\bm e_1\in\R^n$, then $\tilde M_k$ has all distinct eigenvalues of size $<1$ and $\tilde{\bm c}_k^\top \tilde M_k \tilde{\bm k}_k= \bm c_k^\top M_k\bm x_k$. We thus get the same behavior for $n\times n$ matrices.
\informsendofproof \end{proof}


\section{Miscellaneous Results and Clarifications} \label{sec:misc}

\begin{lemma}[Farkas lemma]\label{lem:farkas}
Let $B\in \R^{d\times e}, \bm b\in \R^d$. Then exactly one of the following sets is empty:
\begin{enumerate}
    \item $\set{\bm x\in \R^e\st B\bm x = \bm b, \bm x\ge 0}$.
    \item $\set{\bm y\in \R^d\st B^\top \bm y \le 0, \bm b^\top\bm y>0}$.
\end{enumerate}
\end{lemma}

\begin{proposition}
\label{prop:normbound}
$\rho(A)\le \norm Ap$ for any $p\in\R_{\ge 1}$ and $A\in\C^{n\times n}$.
\end{proposition}
\begin{proof}{\informsproof}
If $\lambda\in \C$ is an eigenvalue of $A$ such that $\rho(A) = \abs{\lambda}$ then $\exists\bm u\in\C^n$ such that $\norm{\bm u}p=1$ and $A\bm u=\lambda\bm u$ whence $\norm{A}{p} \ge \norm{A\bm u}{p} = \norm{\lambda \bm u}{p} = \abs{\lambda} = \rho(A)$.
\informsendofproof \end{proof}

\textbf{All different eignevalues assumption}: Note that for a GAS system the highest magnitude eigenvalue determines the rate of convergence to $\bm 0$. Thus, if the there are two eigenvalues exactly same with the highest magnitude, perturbing one of them so that magnitude reduces by $\epsilon$ for arbitrarily small $\epsilon$ does not affect the convergence to $\bm 0$ or the rate of convergence. The other non-highest magnitude eignevalues can also be perturbed similarly with no effect on the convergence to $\bm 0$ or the rate of convergence.

\textbf{How do $u_i,v_i,w_j$ only depend on $\bm c,\bm x, P$?}
Recall $M=PJP^{-1}$ where $J$ is the real Jordan form of $M$ and $P$ is a real invertible matrix. The expression $\inner{\bm c}{M^t\bm x} = \bm c^\top P J^t P^{-1}\bm x$. Consider $\bm \sigma \sett P^\top \bm c$, $\bm \tau\sett P^{-1}\bm x$. Then $\ds\inner{\bm c}{M^t\bm x} = \bm \sigma^\top J^t \bm \tau = \begin{bmatrix}
\sigma_1\\
\vdots\\
\sigma_{2q}\\
\sigma_{2q+1}\\
\vdots\\
\sigma_{2q+p}
\end{bmatrix}^\top
\begin{bmatrix}
r_1^tJ_1^t\\
&\ddots& & \text{\Large 0} \\
&&r_q^t J_q^t\\
&&&\lambda_1^t\\
&\text{\Large 0} &&&\ddots\\
&&&&&\lambda_p^t
\end{bmatrix}
\begin{bmatrix}
\tau_1\\
\vdots\\
\tau_{2q}\\
\tau_{2q+1}\\
\vdots\\
\tau_{2q+p}
\end{bmatrix} = \sum_{i=1}^q r_i^t \begin{bmatrix} \sigma_{2i-1}\\\sigma_{2i}\end{bmatrix}^\top J_i^t \begin{bmatrix} \tau_{2i-1} \\ \tau{2i} \end{bmatrix} + \sum_{j=1}^p \lambda_j^t \sigma_{2q+j}\tau_{2q+j}$.
Each summand in the first summation looks like $\ds r_i^t\begin{bmatrix} \sigma_{2i-1}\\\sigma_{2i}\end{bmatrix}^\top J_i^t \begin{bmatrix} \tau_{2i-1} \\ \tau{2i} \end{bmatrix} = r_i^t \left(\cos t\theta_i (\tau_{2i-1}\sigma_{2i-1} + \tau_{2i}\sigma_{2i}) - \sin t\theta_i (\tau_{2i}\sigma_{2i-1} - \tau_{2i-1}\sigma_{2i})\right)$. Defining $u_i\sett \tau_{2i-1}\sigma_{2i-1} + \tau_{2i}\sigma_{2i}, v_i\sett \tau_{2i}\sigma_{2i-1} - \tau_{2i-1}\sigma_{2i}, w_j \sett \tau_{2q+j}\sigma_{2q+j}$ for each $1\le i\le q, 1\le j\le p$ gives $\inner{\bm c}{M^t\bm x} = \sum\limits_{i=1}^qr_i^t(u_i\cos t\theta_i-v_i\sin t\theta_i) +\sum\limits_{j=1}^p w_j\lambda_j^t$. Moreover, the $u_i,v_i,w_j$'s are defined in terms of $\bm \sigma, \bm \tau$ which only depend on $P,\bm c,\bm \tau$ which all have only real entries.

\textbf{Detailed comparison to~\cite{blanchet2019quantifying}}:
The simplified dual formulation in Remark 1 of~\cite{blanchet2019quantifying}, when applied to our problem with the natural underlying distance between time as $d(t,t') = |t-t'|$ looks like  
$$\min_{\lambda \geq 0} \lambda \xi + \sum_{t' \mathcal{T}} {\bm p_{t'}} C(\lambda, t') \quad \mbox{ where } \quad C(\lambda, t') = \max_{t \in \mathcal{T}} \langle \bm c,M^t \bm x_0 \rangle - \lambda |t-t'|$$
Note that if $\xi$
 is larger than the maximum possible $|t-t'|$ 
 (which is when~\ref{eq:dre}
 becomes the 
 problem~\ref{eq:re}) then one can see that $\lambda$
 will be zero and the dual form objective above is 
$\sum_{t' \mathcal{T}} {\bm p_{t'}} C(0, t')$. Observe that $C(0, t')$
 does not depend on 
$t'$, thus, this dual becomes $\max_{t \in \mathcal{T}} \langle \bm c,M^t \bm x_0 \rangle$, which is exactly same as~\ref{eq:re}. Thus, this dual formulation does not afford any computational advantage over what we do.

\section{Special case for $n=2$} \label{sec:specialcase}

In this special case for the infinite support, we can provide a better bound on $t_{0}$.
Recall that if $M$ has purely imaginary eigenvalues with the angle parameter of its Jordan form given by $\theta$, $r = \det M>0$, $\inner{\pmb c}{M^t\pmb x} = r^t(a_1\cos t\theta-a_2\sin t\theta)$ where $a_1,a_2$ are completely determined by $\pmb c,\pmb x, M$. We let $d>0,\kappa>0,\alpha\in(0,2\pi), \gamma\in(0,2\pi)$ be such that $a_1+ia_2 = de^{i\alpha}$ and $\ln r+i\theta=\kappa e^{i\gamma}$. We  \textbf{assume} $\alpha,\theta,\gamma\in\Q$.
\begin{lemma} \label{lem:specialcase}
If $\theta \in (0,\pi)$ then for $\ds t_{0} \sett \left\lceil \frac{-\frac\pi2-\alpha+2\pi\cdot \left\lceil \frac{\alpha}{2\pi}+\frac14 \right\rceil}\theta\right\rceil$ we have $g(t_{0})>0$. Consider $\ds\set{x_m\sett\frac{\frac{\pi}{2} - \alpha-\gamma + m\pi}\theta\st m\in \Z}$. Moreover for $\ds m^{*} \sett 1+\left\lceil 
\frac{\theta\log_{r}\pa{\frac{g(t_{0})}{\abs{\sin \gamma} d}}+\alpha+\gamma-\frac\pi2}{\pi} 
\right\rceil$, we have $\ds \sup_{t\in \N} g(t) = \sup_{1\le t\le x_{m^{*}}}g(t)$. 
\end{lemma}


\begin{proof}{\informsproof}
For non-real complex eigenvalues, $J=r\begin{bmatrix}\cost&-\sint\\\sint&\cost\end{bmatrix}$. Then, $J^t=r^t\begin{bmatrix}\cos t\theta&-\sin t\theta\\\sin t\theta&\cos t\theta\end{bmatrix}$ whence $\inner{\pmb c}{M^t\pmb x} = r^t(a_1\cos t\theta-a_2\sin t\theta)$ where $a_1,a_2$ are constants completely determined by the inputs $\pmb c,\pmb x,M$. 

Let's try to see where $r^t(a_1\cos t\theta-a_2\sin t\theta)$ maximizes. 
First extend to a real function $g(x)\sett r^x(a_1\cos x\theta-a_2\sin x\theta) = \Re(e^{ix\theta}(a_1+ia_2))$ so that $g'(x) = r^{x}\Re\pa{e^{i x\theta} \cdot (\ln r + i\theta) \cdot (a_{1}+ia_{2})}$ where $\Re(\cdot)$ denotes the real part of complex numbers. 

Let's name the above complex numbers as $u \sett \ln r + i\theta = \kappa e^{i\gamma}$ and $a\sett a_{1}+ia_{2} = d e^{i\alpha}$ where $\kappa,d > 0$ and $\alpha,\gamma\in\lp0,2\pi\rp$. But $r\in (0,1), \theta\neq\pi$ means that $\gamma \in \pa{\frac\pi2,\frac{3\pi}2}$ and $\sin\gamma\ne \frac\pi \kappa$. With this we have $g(x) = r^x\cos(x\theta+\alpha)$. We assumed $\alpha,\theta\in\Q$.

Now we find all those $x$ such that $g'(x)=0$ and $g(x)>0$ (such an $x$ automatically satisfies $g''(x)\le 0$). Let's first focus on the solutions of $g'(x)=0$, that is, all $x$ satisfying $\Re\pa{e^{i\pa{x\theta + \alpha + \gamma}}}  = \frac{1}{\kappa d}\Re\pa{e^{i x\theta} ua} = 0$. 
The solutions are $\ds\set{x_m\sett\frac{\frac{\pi}{2} - \alpha-\gamma + m\pi}\theta\st m\in \Z}$.  
They satisfy $$g(x_{m}) = (-1)^{m} r^{x_{m}}  d\sin\gamma
.$$ 
Thus $\set{\abs{g(x_{m})}}_{m\ge 0}$ is a strictly decreasing sequence in terms of $m$. In fact, the local maxima over the positive reals are exactly the points $\set{x_{2m}}_{m\in\N_{0}}$. The issue in finding $\ds\sup_{t\in \N}g(t)$ is that these $x_{m}$'s are merely real numbers. 

However, if we could find some $t_{0}\in \N$ such that $g(t_{0})>0$, taking $\epsilon=g(t_{0})$ gives an $N$ such that $g(x_{m})<\epsilon\fa m > N$, whence the required supremum occurs only over $\N\cap[1,x_{N-1}]$.  Now take $t_0$ to be as in the statement of the lemma.

For brevity let's call $p\sett \frac{\alpha}{2\pi}+\frac14$ and $q\sett\lceil p\rceil$ (this choice of $q$ makes $t_0>0$). Clearly $t_{0}\theta +\alpha \ge 2\pi q-\frac\pi2$. It was earlier assumed that $\alpha,\theta\in\Q$ so the last inequality is strict. This means that $\ds \frac{-\frac{\pi}{2}-\alpha+2\pi q}{\theta}\notin\Z$.
\\The following thus holds:
\begin{align*}
\frac{-\frac\pi2-\alpha+2\pi q}{\theta}< t_{0} &\le \frac{-\frac\pi2-\alpha+2\pi q}\theta + 1\\
\implies 2\pi q-\frac \pi2 < t_{0}\theta + \alpha &\le -\frac{\pi}{2} + 2\pi q + \theta \stackrel{[\because~\theta<\pi]}{<} 2\pi q + \frac{\pi}{2}.
\end{align*}
It thus stands that $\Re\pa{e^{it_{0}\theta}\cdot a} = d\cos(t_{0}\theta+\alpha)>0$ which has the same sign as $g(t_{0})$.

For the latter part, notice that if $\ds m \ge m^{*}$ then $x_m > \log_r\left(\frac{g(t_0)}{d\abs{\sin\gamma}}\right)$ whence $\ln\abs{g(x_{m^*})} \le \ln \abs{g(x_{m})} = x_{m} \underbrace{\ln r}_{<0} + \ln(d\abs{\sin\gamma}) < \log_r\left(\frac{g(t_0)}{d\abs{\sin\gamma}}\right) \ln r + \ln(d\abs{\sin\gamma}) = \ln\pa{g(t_{0})}$. So $t > x_{m^{*}}$ implies that $g(t) \le \abs{g(x_{m^{*}})} < g(t_{0})$. The first inequality here is because if $x\ge x_{k}$ is a real number then $\abs{g(x)} = r^x \abs{\cos(x\theta+\alpha)} \le r^{x}  \le r^{x_{k}} \cdot 1 = r^{x_{k}} \cdot \abs{\cos(x_{k}\theta+\alpha)} =  \abs{g(x_k)}$.
\informsendofproof \end{proof}

\section{Additional Experiments}
\label{sec:appendixexperiments}

\subsection{Computational Advantage of Converting Markov chain to GAS}

By reducing the matrix $M$ to $\overline M$, we reduce its size from 
$n \times n$ to $(n-1)\times(n-1)$
 while keeping all the essential information intact. Asymptotically, this is of the same order but practically, it has consequences in compute times. To show a practical advantage of using 
 $\overline M$ instead of $M$, we ran the following experiment. Set parameters $k=5000, n=75, T=10^{10}$. We generate a matrix $M$
 of size $n \times n$ randomly $k$
 times, reduce it to $\overline M = AMB$
 (with a smaller size as in Section~\ref{sec:relation} of our paper); then compute the matrix powers $M^T$
 and $\overline M^T$. We look at each individual run and record what percentage of those $k$ runs has a lower running time for 
 $\overline M$
 and also look at their respective total running times. Let the total times (for $k$
 runs) be $t_1,t_2$ 
 for $M, \overline M$
 respectively. We repeat this entire experiment 5 times independently and then report these percentages and the percentage decrease from $t_1$
 to reach $t_2$, namely $x = \frac{t_1 - t_2}{t_1}*100$. In other words, computation of 
 $\overline M^T$
 takes $x$\% less time than the computation of 
 $M^T$. All percentages are rounded to the nearest integer. This is shown in the table below.

 \begin{table}[t!]
     \centering
     \begin{tabular}{ccc}
     \toprule
    Experiment run	& \% won &	\% reduction in time \\
    \midrule
    1	& 79 & 25\\
2	& 77 & 31 \\
3	& 75 & 44 \\
4	& 75 &  31 \\
5	& 77 & 37 \\
\bottomrule
     \end{tabular}
     \label{tab:advantage}
 \end{table}
\end{document}